\documentclass[11pt,british]{extarticle}
\usepackage{mathptmx}
\usepackage[T1]{fontenc}
\usepackage[latin9]{inputenc}
\usepackage[a4paper]{geometry}
\usepackage{color}
\usepackage{babel}
\usepackage{float}
\usepackage{mathrsfs}
\usepackage{mathtools}
\usepackage{amsmath}
\usepackage{amsthm}
\usepackage{amssymb}
\usepackage{graphicx}
\usepackage[numbers]{natbib}
\usepackage[unicode=true,pdfusetitle,
 bookmarks=true,bookmarksnumbered=true,bookmarksopen=false,
 breaklinks=false,pdfborder={0 0 1},backref=false,colorlinks=true]
 {hyperref}
\hypersetup{
 pdfborderstyle=,linkcolor=black,citecolor=blue,urlcolor=black,filecolor=blue,pdfpagelayout=OneColumn,pdfnewwindow=true,pdfstartview=XYZ,plainpages=false}

\makeatletter
\numberwithin{equation}{section}
\theoremstyle{plain}
\newtheorem{thm}{\protect\theoremname}[section]
\theoremstyle{plain}
\newtheorem{lem}[thm]{\protect\lemmaname}
\theoremstyle{definition}
\newtheorem{defn}[thm]{\protect\definitionname}

\@ifundefined{date}{}{\date{}}
\usepackage{babel}
\usepackage{caption}
\usepackage[nottoc]{tocbibind}
\allowdisplaybreaks[4]
\usepackage{dsfont}
\DeclareMathAlphabet{\mathcal}{OMS}{cmsy}{m}{n}

\providecommand{\lemmaname}{Lemma}

\providecommand{\theoremname}{Theorem}

\makeatother

\providecommand{\definitionname}{Definition}
\providecommand{\lemmaname}{Lemma}
\providecommand{\theoremname}{Theorem}

\begin{document}

\title{Random vortex and expansion-rate model for \\ Oberbeck-Boussinesq fluid flows}

\author{Zihao Guo\thanks{Zhongtai Securities Institute for Financial Studies, Shandong University, Jinan, China, 250100, Email:
\protect\href{mailto:gzhsdu@mail.sdu.edu.cn}{gzhsdu@mail.sdu.edu.cn}}, \ Zhongmin Qian\thanks{Mathematical Institute, University of Oxford, Oxford, United Kingdom, OX2 6GG, and Oxford Suzhou Centre for Advanced Research, Suzhou, China.  Email:
\protect\href{mailto:qianz@maths.ox.ac.uk}{qianz@maths.ox.ac.uk}}, and Zihao Shen\thanks{Mathematical Institute, University of Oxford, Oxford, United Kingdom, OX2 6GG, Email: 
\protect\href{mailto:zihao.shen@bnc.ox.ac.uk}{zihao.shen@bnc.ox.ac.uk}}}
\maketitle
\begin{abstract}
By using a formulation of a class of compressible viscous flows with a heat source via vorticity and expansion-rate, we study the Oberbeck-Boussinesq flows. To this end we establish a new integral representation for solutions of
parabolic equations subject to certain boundary condition, which allows us to develop a random vortex method for certain
compressible flows and to compute numerically solutions of their dynamical models. Numerical experiments are carried
out, which not only capture detailed B\'enard convection but also are capable of providing additional information on the
fluid density and the dynamics of expansion-rate of the flow.

\medskip{}


\medskip{}


\end{abstract}

\section{Introduction}

The random vortex method is a useful tool for simulating incompressible viscous flows (cf. \citep{Cottet and Koumoutsakos 2000,Majda and Bertozzi 2002},
also \citep{Constantin2001a,Constantin2001b,ConstantinIyer2011}, \citep{Bunello1999,Busnello2005}, \citep{Drivas2017a},
\citep{EyinkGuptaZaki2020a}, \citep{Chorin1973,Chorin1980} and etc. for details). The method has been generalized to wall-bounded incompressible viscous
flows (cf. \citep{Drivas2017}, \citep{EyinkGuptaZaki2020b} and \citep{LiQianXu2023,Z.Qian 2022}). In a recent work \citep{Z.Qian and Z.Shen}, 
a formulation via the vorticity and the expansion-rate of (possibly compressible) viscous flows has been established, which
has been shown to be equivalent to the Navier-Stokes equations for a class of (compressible) flows (cf. \citep{Landau-LifshitzFluids}).
This vorticity and expansion-rate formulation allows to develop a random vortex and expansion-rate method to simulate certain
compressible flows. Due to the highly nonlinear and complicated nature of compressible flows observed in nature, when applying
the general theory to compressible flows, certain approximation procedure must be applied first before Monte-Carlo simulations
may be carried out. In the present paper, we consider a compressible flow demonstrating the phenomena of B\'enard's convection
(cf. \citep{Benard1900,Boussinesq1903,Drazin-Reid1981}), applying the vortex and expansion-rate formulation to the fluid
dynamical equations describing B\'enard's convection we propose an approximate compressible flow model, which, in contrast
to the Oberbeck-Boussinesq flows, retains the property of compressibility. We then employ the random vortex and expansion-rate approach
to this model and derive implicit stochastic integral representations for solutions to the approximate model in the spirit
of McKean-Vlasov stochastic differential equations. The stochastic representations may be used to carry out numerical experiments
to demonstrate the small variation of the fluid density and the expansion-rate which measures the compressibility, the information
which is missing in the Oberbeck-Boussinesq model.

The paper is organized as follows. In Section 2, we review essential preliminary knowledge related to vorticity and the expansion-rate
formulation of a viscous flow, alongside a modified version of the Oberbeck-Boussinesq equations. Subsequently, in Section
3, we propose an integral representation for parabolic equations. To represent velocity and temperature, the Biot-Savart
law is introduced in Section 4, followed by the definition of three Brownian particles corresponding to the coupled stochastic
differential equations outlined in the representation theorem of Section 3. Utilizing the representation formula in Section
3 and the Biot-Savart law in Section 4, we establish representations for the dynamical variables of the Oberbeck-Boussinesq
model in $\mathbb{R}_{+}^{d}$, which are delineated in Sections 5. Finally, we
present numerical experiments of Oberbeck-Boussinesq flow with heat source applying from the boundary under wall-bounded domain in Section 6.

\section{The vorticity and expansion-rate function formulation}

In this section, we recall the vorticity and expansion-rate formulation of a viscous flow (maybe compressible however) past
a solid wall. This formulation (with for a compressible flow constrained in a bounded region $D$ with boundary $\partial D$)
has been proposed in \citep{Z.Qian and Z.Shen} by identifying the (in general non-homogeneous) appropriate boundary conditions.
This formulation can be therefore implemented, and can be regarded as an extension of the random vortex method to compressible
flows.

Let $u(x,t)$, $\rho(x,t)$ and $p(x,t)$ be the velocity, the fluid density and the pressure of a viscous flow, respectively.
Let $\omega=\nabla\wedge u$ be the vorticity and $\phi=\nabla\cdot u$ be the expansion-rate. Therefore $u$ and $\omega$
are (time-dependent) vector fields, and $\rho$ and $p$ are scalar functions. The dynamics of the flow are determined by
the Navier-Stokes equations for compressible flows (cf. Landau-Lifschitz \citep{Landau-LifshitzFluids}): 
\begin{equation}
\partial_{t}\rho+\nabla\cdot(\rho u)=0\label{C1}
\end{equation}
and

\begin{equation}
\rho\left(\partial_{t}+u\cdot\nabla\right)u=\mu\Delta u-\nabla(p-\lambda\phi)+\rho F,\label{NS-C1}
\end{equation}
where $\lambda=\zeta+\frac{d-2}{d}\mu$ and $F$ is an external force. In this paper, $F$ is enforced to the fluid flow through
a heating source. The the velocity is subject to the non-slip condition that $u(x,t)=0$ for $x\in\partial\varOmega$ and $t>0$.

We consider only Newton flows where the viscosity coefficients $\mu$, $\zeta$ are constant. Then the vorticity $\omega$
and expansion-rate $\phi$ evolve according to the following equations (cf. \citep{Z.Qian and Z.Shen} for the details of
the derivation): 
\begin{align}
\partial_{t}\phi & =\sum_{i,j}\frac{\partial}{\partial x^{j}}\left(\frac{\mu+\lambda}{\rho}\delta_{ij}\frac{\partial}{\partial x^{i}}\phi\right)-(u\cdot\nabla)\phi-\nabla\rho^{-1}\cdot\nabla p\nonumber \\
 & -\mu\nabla\rho^{-1}\cdot(\nabla\wedge\omega)+|\omega|^{2}-|\nabla u|^{2}-\frac{1}{\rho}\Delta p+\nabla\cdot F.\label{phi-1}
\end{align}
and 
\begin{align}
\partial_{t}\omega & =\sum_{i,j}\frac{\partial}{\partial x^{j}}\left(\frac{\mu}{\rho}\delta_{ij}\frac{\partial}{\partial x^{i}}\omega\right)-(u\cdot\nabla)\omega+(\omega\cdot\nabla)u-\phi\omega\nonumber \\
 & -\mu\frac{\partial\rho^{-1}}{\partial x^{j}}\nabla\omega^{j}-\nabla\rho^{-1}\wedge\nabla p+(\lambda+\mu)\nabla\rho^{-1}\wedge\nabla\phi+\nabla\wedge F.\label{om-1}
\end{align}
For 2-dimensional flows, $(\omega\cdot\nabla)u$ vanishes identically, and the vorticity transport equation is simplified as followed:
\begin{align}
\partial_{t}\omega & =\sum_{i,j}\frac{\partial}{\partial x^{j}}\left(\frac{\mu}{\rho}\delta_{ij}\frac{\partial}{\partial x^{i}}\omega\right)-(u\cdot\nabla)\omega-\phi\omega\nonumber \\
 & -\mu\frac{\partial\rho^{-1}}{\partial x^{j}}\nabla\omega^{j}-\nabla\rho^{-1}\wedge\nabla p+(\lambda+\mu)\nabla\rho^{-1}\wedge\nabla\phi+\nabla\wedge F.\label{2D-vort}
\end{align}

These evolution equations together with appropriate boundary conditions for $\phi$ and $\omega$ are equivalent to the Navier-Stokes
equations, based on which we may develop a version of the random vortex method for compressible flows. While these equations
are rather complicated, it is desirable to work with simplified equations by taking into specific features of flows in question.

Let us consider a viscous flow past a solid wall with a small fluid density gradient, which means that the flow is nearly incompressible,
known as Oberbeck-Boussinesq flows. Such flows are typically approximated as incompressible flows and studied in existing literature via the incompressible Navier-Stokes equations.
However, in our paper, we use a simplified version of the vortex and expansion-rate formulation (\ref{phi-1}, \ref{om-1})
to implement numerical schemes of compressible flows via the random vortex method. To state our simplified equations, we observe that for incompressible
flows (\ref{phi-1}) leads to the following constraint: 
\[
-\nabla\rho^{-1}\cdot\nabla p-\mu\nabla\rho^{-1}\cdot(\nabla\wedge\omega)+|\omega|^{2}-|\nabla u|^{2}-\frac{1}{\rho}\Delta p+\nabla\cdot F=0
\]
so that the terms on the left-hand side will be omitted in (\ref{phi-1}). Similarly, since the viscosity coefficients and
the gradient of the density $\nabla\rho$ are small, so that the following terms 
\[
-\mu\frac{\partial\rho^{-1}}{\partial x^{j}}\nabla\omega^{j}-\nabla\rho^{-1}\wedge\nabla p+(\lambda+\mu)\nabla\rho^{-1}\wedge\nabla\phi
\]
are ignored in (\ref{om-1}) or (\ref{2D-vort}).

Therefore we consider the following system which is a variation of the Oberbeck-Boussinesq equations: 
\begin{equation}
\partial_{t}\rho+u\cdot\nabla\rho=-\phi\rho,\label{AA-c}
\end{equation}
\begin{equation}
%\partial_{t}\phi+(u\cdot\nabla)\phi=\sum_{i,j}\frac{\partial}{\partial x^{j}}\left(\frac{\mu+\lambda}{\rho}\delta_{ij}\frac{\partial}{\partial x^{i}}\phi\right),\label{AA-2}
\partial_{t}\phi+(u\cdot\nabla)\phi=\sum_{i,j}\frac{\partial}{\partial x^{j}}\left(\frac{\mu+\lambda}{\rho}\delta_{ij}\frac{\partial}{\partial x^{i}}\phi\right),\label{AA-2}
\end{equation}
\begin{equation}
\partial_{t}\omega+(u\cdot\nabla)\omega=\sum_{i,j}\frac{\partial}{\partial x^{j}}\left(\frac{\mu}{\rho}\delta_{ij}\frac{\partial}{\partial x^{i}}\omega\right)-\phi\omega+\nabla\wedge F(\theta),\label{AA-3}
\end{equation}
\begin{equation}
\partial_{t}\theta+(u\cdot\nabla)\theta=\frac{\kappa}{\rho_{0}}\Delta\theta+g\label{AA-4}
\end{equation}
and 
\begin{equation}
\Delta u=-\nabla\wedge\omega+\nabla\phi\label{AA-5}
\end{equation}
in $D\times[0,\infty)$, where $\kappa>0$ is heat conducting constant, $\theta$ represents the temperature and $g$ represents
the heating resource to the fluid. The representation theorem of temperature involves the gradient of temperature equation, so under the exist assumption, we need to approximate the temperature equation by substituting $\rho$ with initial density $\rho_0$. Further more, $\rho_{0}\in C^{1}(\mathbb{R}^{d})$, we can extend it to be $$\rho_{0}(x,t)=\rho(x,0)\quad(x,t)\in\mathbb{R}^{d}\times[0,\infty).$$
In practice we usually take $\rho_{0}$ to be a constant.

The boundary conditions are set as the following: 
\begin{equation}
u=0,\quad\phi=\nabla\cdot u,\quad\nabla\cdot\omega=0,\quad(\omega-\nabla\wedge u)^{\parallel}=0\quad\textrm{ in }\partial D\label{B-c}
\end{equation}
and 
\begin{equation}
\theta(x,t)=\theta_{b}(x,t)\quad\textrm{ for }x\in\partial D.\label{t-bc}
\end{equation}
The initial state is determined by the initial velocity $u_{0}$ and the initial density $\rho_{0}$, and 
\begin{equation}
\rho(\cdot,0)=\rho_{0},\quad\phi(\cdot,0)=\phi_{0}=\nabla\cdot u_{0},\quad\omega(\cdot,0)=\omega_{0}=\nabla\wedge u_{0}\quad\textrm{ and }\theta(\cdot,0)=\theta_{0}\label{Int-1}
\end{equation}
where $u_{0}$, $\rho_{0}$ and $\theta_{0}$ are given.

The well-posededness of the previous system of PDEs is however not studied in the present paper, and will be published in
a separate work. However we would like to mention the following fact. 
\begin{lem}
If $\rho,\phi,\omega,\theta$ and $u$ are solutions to the above initial and boundary problem, then 
\[
\phi=\nabla\cdot u\quad\textrm{ and }\omega=\nabla\wedge u\quad\textrm{ in }D.
\]
\end{lem}

As we have mentioned before, for 2-dimensional flows, $(\omega\cdot\nabla)u$ appearing in (\ref{AA-3}) equals zero and the boundary
conditions in (\ref{B-c}) can be simplified: 
\begin{equation}
u=0,\quad\phi=\nabla\cdot u,\quad\omega=\frac{\partial}{\partial x^{1}}u^{2}-\frac{\partial}{\partial x^{2}}u^{1}\quad\textrm{ in }\partial D\label{B-c-1}
\end{equation}

\section{An integral representations for parabolic equations}

The random-vortex method for the model (\ref{AA-c}, \ref{AA-2}, \ref{AA-3}, \ref{AA-4}, \ref{AA-5}) is derived by using
a new representation theorem for solutions to some parabolic equations which will be established in this section.

To this end we need a few facts about the diffusion process with with the infinitesimal generator $\mathscr{L}_{a;b}=\mathscr{L}_{a;b,0}$,
where 
\[
\mathscr{L}_{a;b,c}=\sum_{i,j=1}^{d}\frac{\partial}{\partial x^{j}}a^{i,j}(x,t)\frac{\partial}{\partial x^{i}}+\sum_{i=1}^{d}b^{i}(x,t)\frac{\partial}{\partial x^{i}}+c(x,t)
\]
where $a(x,t)=(a^{ij}(x,t))_{i,j\leq d}$ is a smooth $d\times d$ symmetric matrix-valued function on $\mathbb{R}^{d}\times[0,\infty)$
with bounded derivatives, satisfying the uniformly elliptic condition. $b(x,t)=(b^{1}(x,t),\cdots,b^{d}(x,t))$ is a smooth
(time-dependent) vector field, and $c(x,t)$ is a continuous bounded scalar function. Let $\mathbb{P}_{a;b}^{\xi,s}$ be the
diffusion measure with the infinitesimal generator $\mathscr{L}_{a;b}$ started from $\xi\in\mathbb{R}^{d}$ at instance
$s\geq0$, and $\mathbb{P}_{a;b}^{\xi,0\rightarrow\eta,T}$ denotes the conditional law $\mathbb{P}_{a;b}^{\xi,0}\left[\cdot\right]$
conditional on that the diffusion stopped at the location $\eta\in\mathbb{R}^{d}$ at time $T>0$, for a precise definition
cf. \citep{Dellacherie and Meyer Volume D,Z.Qian and Z.Shen}.

Let $\Sigma_{a;b,c}$ denotes the associated fundamental solution to operator $\mathscr{L}_{a;b,c}-\partial_{t}$ and $\Sigma_{a;b,c}^{*}$
to operator $\mathscr{L}_{a;b,c}+\partial_{t}$, and $p_{a,b}(s,x;t,y)$ denotes the transition probability density function
of the diffusion with generator $\mathscr{L}_{a;b}=\mathscr{L}_{a;b,0}$. It is known that $p_{a,b}=\Sigma_{a;b}^{*}$, cf.
\citep{Friedman 1964,Stroock and Varadhan 1979}.

The following duality of conditional diffusion laws has been established in \citep{Z.Qian 2022,Z.Qian and Z.Shen}. 
\begin{thm}
\label{thm3.1}Assume that $\nabla\cdot b$ is bounded. Let $T>0$ and $\xi,\eta\in\mathbb{R}^{d}$. Let $\tau_{T}$ denote
the time-reversal which maps a path $w(\cdot)$ to $w(T-\cdot)$. Then

1) $\mathbb{P}_{a^{T};-b^{T}}^{\eta,0\rightarrow\xi,T}\circ\tau_{T}$ and $\mathbb{P}_{a;b}^{\eta,0\rightarrow\xi,T}\circ\tau_{T}$
are absolutely continuous with each other, and 
\begin{equation}
\frac{\mathrm{d}\mathbb{P}_{a^{T};-b^{T}}^{\eta,0\rightarrow\xi,T}\circ\tau_{T}}{\mathrm{d}\mathbb{P}_{a;b}^{\xi,0\rightarrow\eta,T}}=\frac{\textrm{e}^{\int_{0}^{T}\nabla\cdot b(X_{s},s)\textrm{d}s}}{\mathbb{P}_{a;b}^{\xi,0\rightarrow\eta,T}\left[\textrm{e}^{\int_{0}^{T}\nabla\cdot b(X_{s},s)\textrm{d}s}\right]}\label{eq:Dual1}
\end{equation}
where $f^{T}(x,t)=f(x,T-t)$ for $t\leq T$ and $f^{T}(x,t)=f(x,0)$ for $t\geq T$, and $X=(X_{t})_{t\geq0}$ donates the
coordinate process.

2) It holds 
\begin{equation}
\frac{\varSigma_{a;b,c}^{\star}(\xi,\tau;x,t)}{p_{a;b}(\tau,\xi;t,x)}=\mathbb{P}_{a;b}^{\xi,\tau\rightarrow x,t}\left[\textrm{e}^{\int_{\tau}^{t}c(X_{s},s)\textrm{d}s}\right]\label{eq:Dual2}
\end{equation}
for $0\leq\tau<t\leq T$, $\xi,x\in\mathbb{R}^{d}$. 
\end{thm}

We are now in a position to establish the integral representation to the solutions of the following initial and boundary
problem to the following (linear) parabolic system 
\begin{equation}
\left(\mathscr{L}_{a;b,c}-\partial_{t}\right)\varPsi^{i}+q_{j}^{i}\varPsi^{j}+f^{i}=0\quad\mathrm{in}\;D\times(0,\infty),\quad\varPsi(\cdot,0)=\varPsi_{0},\label{01-t}
\end{equation}
(where $D\subset\mathbb{R}^{d}$) subject to the boundary value $\varPsi^{i}(x,t)=0$ for $x\in\partial D$ and $t>0$, where
$i=1,\ldots,d$. The multiplier $q(x,t)=(q^{ij}(x,t))$ is an $n\times n$ square-matrix valued bounded and continuous function.

For simplicity we extend the definition $q(x,t)$ to all $x\in\mathbb{R}^{d}$ with the convention that $q(x,t)=0$ for $x\in D^{c}$.

Here and thereafter Einstein's convention that repeated indices are summed up is applied. 
\begin{thm}
\label{thm3.2} For $\psi\in C([0,T];\mathbb{R}^{d})$, $A(\psi,T;t)$ denotes the solution to the following system 
\begin{equation}
\frac{\textrm{d}}{\textrm{d}t}A_{j}^{i}(\psi,T;t)=-A_{k}^{i}(\psi,T;t)q_{j}^{k}(\psi(t),t)-A_{j}^{i}(\psi,T;t)c(\psi(t),t)\label{A-1}
\end{equation}
such that $A_{j}^{i}(\psi,T;T)=\delta_{ij}$ for $i,j\leq n$. Then 
\begin{align*}
\varPsi(\xi,T) & =\int_{D}\mathbb{E}\left[\left.1_{\{T<\zeta(X^{\eta})\}}\textrm{e}^{-\int_{0}^{T}\nabla\cdot b(X_{s},s)\textrm{d}s}A(X,T;0)\varPsi_{0}(\eta)\right|X_{T}^{\eta}=\xi\right]p_{a;-b}(0,\eta;T,\xi)\textrm{d}\eta\\
 & +\int_{0}^{T}\int_{D}\mathbb{E}\left[\left.1_{\{t>\gamma_{T}(X^{\eta})\}}\textrm{e}^{-\int_{0}^{T}\nabla\cdot b(X_{s},s)\textrm{d}s}A(X,T;t)f(X_{t},t)\right|X_{T}^{\eta}=\xi\right]p_{a;-b}(0,\eta;T,\xi)\textrm{d}\eta\textrm{d}t
\end{align*}
for any $T>0$ and $\xi\in D$, where $X^{\eta}$ is a $\mathscr{L}_{a,-b}$-diffusion started from $\eta$ at
time $0$. Here $(A\psi)^{i}=A_{j}^{i}\psi^{j}$, $\zeta(w)=\inf\left\{ t\geq0:w(t)\notin D\right\} $ and $\gamma_{T}(w)=\sup\left\{ s\in(0,T):w(s)\in\partial D\right\} $. 
\end{thm}

\begin{proof}
Let $T>0$ be arbitrary but fixed in the discussion below. For each $T>0$, let $\tilde{X}^{\xi}$ be the solution of the
following stochastic differential equation: 
\[
\mathrm{d}(\tilde{X}_{t}^{\xi})^{i}=(b^{i}+\nabla\cdot a^{i})(\tilde{X}_{t}^{\xi},T-t)\mathrm{d}t+\sum_{k=1}^{d}\sqrt{2}\sigma_{k}^{i}(\tilde{X}_{t}^{\xi},T-t)\mathrm{d}B_{t}^{k},\quad\tilde{X}_{0}^{\xi}=\xi
\]
for every $\xi\in\mathbb{R}^{d}$ and $B$ is a standard Brownian motion in $\mathbb{R}^{d}$ on some probability space,
$b(x,t)=0$ for $t<0$, and $a^{ij}(x,t)=\sigma_{k}^{i}(x,t)\sigma_{j}^{k}(x,t).$ Define $\tilde{A}(t)=\tilde{A}(w,t)$
by solving the differential equations 
\[
\mathrm{d}\tilde{A}_{j}^{i}(t)=\tilde{A}_{k}^{i}(t)(q_{j}^{k}+c^{i})(w(t),T-t)\mathrm{d}t,\quad\tilde{A}_{j}^{i}(0)=\delta_{ij}.
\]
Let $Y_{t}=\varPsi(\tilde{X}_{t\wedge T_{\xi}}^{\xi},T-t)=1_{\{t<T_{\xi}\}}\varPsi(\tilde{X}_{t}^{\xi},T-t)$, where $T_{\xi}=\mathrm{inf}\left\{ t:\tilde{X}_{t}^{\xi}\notin D\right\} $.
Let $M^{i}=\tilde{A}_{j}^{i}Y^{j}.$ Using Itô's formula we obtain that 
\[
M_{t}^{i}=Y_{0}^{i}+\sqrt{2}\sum_{i,k=1}^{d}\int_{0}^{t}1_{\{s<T_{\xi}\}}\frac{\partial\varPsi^{j}(\tilde{X}_{s}^{\xi},T-s)}{\partial x^{i}}\sigma_{k}^{i}(\tilde{X}_{s}^{\xi},T-s)\mathrm{d}B_{s}^{k}-\int_{0}^{t}1_{\{s<T_{\xi}\}}\tilde{A}_{j}^{i}(s)f^{j}(\tilde{X}_{s}^{\xi},T-s)\mathrm{d}s,
\]
where $\tilde{A}(s)=\tilde{A}(\tilde{X},s)$. Taking expectation on both sides, since $\varPsi$ vanishes along boundary
$\partial D$, we have 
\begin{align*}
\varPsi^{i}(\xi,T) & =\mathbb{E}\left[\tilde{A}_{j}^{i}(T)\varPsi_{0}^{j}(\tilde{X}_{T})1_{\{T<T_{\xi}\}}\right]+\int_{0}^{T}\mathbb{E}\left[\tilde{A}_{j}^{i}(s)f^{j}(\tilde{X}_{s}^{\xi},T-s)\right]\textrm{d}s\\
 & \equiv J_{1}^{i}+J_{2}^{i}.
\end{align*}
By conditional expectation, we may write 
\begin{align*}
J_{1}^{i}(\xi,T) & =\int_{D}\mathbb{E}\left[\left.\tilde{A}_{j}^{i}(T)1_{\{T<T_{\xi}\}}\right|\tilde{X}_{T}^{\xi}=\eta\right]\varPsi_{0}^{j}(\eta)p_{a^{T},b^{T}}(0,\xi,T,\eta)\textrm{d}\eta,
\end{align*}
where $p_{a^{T},b^{T}}(\tau,\xi;T,\eta)$ is the transition probability density function of the diffusion $\tilde{X}^{\xi}$.
Similarly 
\[
J_{2}^{i}(\xi,T)=\int_{0}^{T}\int_{\mathbb{\mathbb{R}}^{d}}\mathbb{E}\left[\left.\tilde{A}_{j}^{i}(t)f^{j}(\tilde{X}_{t}^{\xi},T-t)1_{\{t<T_{\xi}\}}\right|\tilde{X}_{T}^{\xi}=\eta\right]p_{a^{T},b^{T}}(0,\xi;T,\eta)\mathrm{d}\eta\textrm{d}t.
\]
On the other hand 
\begin{align*}
p_{a^{T};b^{T}}(0,\xi;T,\eta) & =\varSigma_{a;-b,-\nabla\cdot b}^{\star}(\eta,0;\xi,T)\\
 & =p_{a;-b}(0,\eta;T,\xi)\mathbb{P}_{a;-b}^{\eta,0\rightarrow\xi,T}\left[\textrm{e}^{-\int_{0}^{T}\nabla\cdot b(X_{s},s)\textrm{d}s}\right]
\end{align*}
where the second equality follows from (\ref{eq:Dual2}). By substituting this into the previous representation we obtain
\begin{align*}
J_{1}^{i}(\xi,T) & =\int_{D}\mathbb{E}\left[\left.\tilde{A}_{j}^{i}(T)1_{\{t<T_{\xi}\}}\right|\tilde{X}_{T}^{\xi}=\eta\right]\varPsi_{0}^{j}(\eta)p_{a;-b}(0,\eta;T,\xi)\mathbb{P}_{a;-b}^{\eta,0\rightarrow\xi,T}\left[\textrm{e}^{-\int_{0}^{T}\nabla\cdot b(X_{s},s)\textrm{d}s}\right]\textrm{d}\eta\\
 & =\int_{D}\mathbb{P}_{a^{T},b^{T}}^{\xi,0\rightarrow\eta,T}\left[\tilde{A}_{j}^{i}(T)1_{\{t<T_{\xi}\}}\right]\varPsi_{0}^{j}(\eta)p_{a;-b}(0,\eta;T,\xi)\mathbb{P}_{a;-b}^{\eta,0\rightarrow\xi,T}\left[\textrm{e}^{-\int_{0}^{T}\nabla\cdot b(X_{s},s)\textrm{d}s}\right]\textrm{d}\eta
\end{align*}
and 
\begin{align*}
J_{2}^{i}(\xi,T) & =\int_{0}^{T}\int_{D}\mathbb{E}\left[\left.\tilde{A}_{j}^{i}(t)f^{j}(\tilde{X}_{t}^{\xi},T-t)1_{\{t<T_{\xi}\}}\right|\tilde{X}_{T}^{\xi}=\eta\right]p_{a;-b}(0,\eta;T,\xi)\\
 & \ \ \ \times\mathbb{P}_{a;-b}^{\eta,0\rightarrow\xi,T}\left[\textrm{e}^{-\int_{0}^{T}\nabla\cdot b(X_{s},s)\textrm{d}s}\right]\mathrm{d}\eta\textrm{d}t\\
 & =\int_{0}^{T}\int_{D}\mathbb{P}_{a^{T},b^{T}}^{\xi,0\rightarrow\eta,T}\left[\tilde{A}_{j}^{i}(t)f^{j}(\tilde{X}_{t}^{\xi},T-t)1_{\{t<T_{\xi}\}}\right]p_{a;-b}(0,\eta;T,\xi)\\
 & \ \ \ \times\mathbb{P}_{a;-b}^{\eta,0\rightarrow\xi,T}\left[\textrm{e}^{-\int_{0}^{T}\nabla\cdot b(X_{s},s)\textrm{d}s}\right]\mathrm{d}\eta\textrm{d}t
\end{align*}
By the conditional law duality, Theorem \ref{thm3.1} again, where $X$ is the coordinate process, we therefore conclude
that 
\[
J_{1}^{i}(\xi,T)=\int_{D}\mathbb{P}_{a;-b}^{\eta\rightarrow\xi}\left[\textrm{e}^{-\int_{0}^{T}\nabla\cdot b(w(s),s)\textrm{d}s}1_{\{T<\zeta(w\circ\tau_{T})\}}\tilde{A}_{j}^{i}(w\circ\tau_{T},T)\right]\varPsi^{j}(\eta,0)p_{a;-b}(0,\eta;T,\xi)\textrm{d}\eta
\]
and 
\begin{align*}
J_{2}^{i}(\xi,T) & =\int_{0}^{T}\int_{D}\mathbb{P}_{a;-b}^{\eta\rightarrow\xi}\left[\textrm{e}^{-\int_{0}^{T}\nabla\cdot b(w(s),s)\textrm{d}s}1_{\{t<\zeta(w\circ\tau_{T})\}}\tilde{A}_{j}^{i}(w\circ\tau_{T},t)f^{j}(w(T-t),T-t)\right]\\
 & \ \times p_{a;-b}(0,\eta;T,\xi)\mathrm{d}\eta\textrm{d}t
\end{align*}
Define $A_{j}^{i}(w,T;s)=A_{j}^{i}(w\circ\tau_{T},T-s)$. Then 
\[
J_{1}^{i}(\xi,T)=\int_{D}\mathbb{P}_{a;-b}^{\eta\rightarrow\xi}\left[\textrm{e}^{-\int_{0}^{T}\nabla\cdot b(X_{s},s)\textrm{d}s}1_{\{T<\zeta(w\circ\tau_{T})\}}A_{j}^{i}(\psi,T;0)\right]\varPsi^{j}(\eta,0)p_{a;-b}(0,\eta;T,\xi)\textrm{d}\eta
\]
and 
\begin{align*}
J_{2}^{i}(\xi,T) & =\int_{0}^{T}\int_{D}\mathbb{P}_{a;-b}^{\eta\rightarrow\xi}\left[\textrm{e}^{-\int_{0}^{T}\nabla\cdot b(w(s),s)\textrm{d}s}1_{\{t<\zeta(w\circ\tau_{T})\}}A_{j}^{i}(\psi,T;T-t)f^{j}(\psi(T-t),T-t)\right]\\
 & \ \times p_{a;-b}(0,\eta;T,\xi)\mathrm{d}\eta\textrm{d}t\\
 & =\int_{0}^{T}\int_{D}\mathbb{P}_{a;-b}^{\eta\rightarrow\xi}\left[\textrm{e}^{-\int_{0}^{T}\nabla\cdot b(w(s),s)\textrm{d}s}1_{\{T-t<\zeta(w\circ\tau_{T})\}}A_{j}^{i}(\psi,T;t)f^{j}(\psi(t),t)\right]p_{a;-b}(0,\eta;T,\xi)\mathrm{d}\eta\textrm{d}t
\end{align*}
which imply the integral representation. 
\end{proof}

\section{Random vortex method for compressible flows }

In this section we describe a random vortex method for (compressible) viscous flows by using the Feynman type formula for
solutions of parabolic equations derived in Section 3. To this end we recall that the Green function in $\mathbb{R}^{d}$
is given as

\begin{equation}
\Gamma_{d}(y,x)=\left\{ \begin{aligned} & -\frac{1}{(d-2)s_{d-1}}\frac{1}{|y-x|^{d-2}},\quad & \textrm{ if }\,\,d>2,\\
 & \frac{1}{2\pi}\ln|y-x|,\quad & \textrm{ if }\,\,d=2,
\end{aligned}
\right.\label{Ga-01}
\end{equation}
where $s_{d-1}$ is the surface area of a unit sphere in $\mathbb{R}^{d}$. 

On $\mathbb{R}_{+}^{d}=\left\{ (x_{1},\cdots,x_{d})\in\mathbb{R}^{d}:x_{d}>0\right\} $, there are integral representations for vector fields and functions. Indeed the Green function on $\mathbb{R}_{+}^{d}$ is given by $G_{d}(y,x)=\Gamma_{d}(y,x)-\Gamma_{d}(y,\overline{x})$
for $x\neq y$, $y\neq\overline{x}$, where $\overline{x}=(x_{1},\cdots,x_{d},-x_{d})$ denotes the reflection about the
hyperplane $x_{d}=0$. The Biot-Savart kernel on $\mathbb{R}_{+}^{d}$ is $\Lambda_{d}(y,x)=\nabla_{y}G_{d}(y,x)$, and the
Biot-Savart law says that if $u$ is a vector field which is regular and decays to zero at infinity, then 
\begin{equation}
u(x)=\int_{\mathbb{R}_{+}^{d}}\Lambda_{d}(y,x)\wedge\omega(y)dy-\int_{\mathbb{R}_{+}^{d}}\phi(y)\Lambda_{d}(y,x)\textrm{d}y,\label{BS-e1-1}
\end{equation}
\begin{equation}
\theta(x)=-\int_{\mathbb{R}_{+}^{2}}\Lambda_{2}(y,x)\cdot\Theta(y)\mathrm{d}y+\frac{x_{2}}{\pi}\int_{-\infty}^{\infty}\frac{\theta_{b}(y_{1})}{|y_{1}-x_{1}|^{2}+x_{2}^{2}}\mathrm{d}y_{1}\label{Th-2D-Half}
\end{equation}
if $d=2$ and $x=(x_{1},x_{2})$ with $x_{2}>0$, and 
\begin{equation}
\theta(x)=-\int_{\mathbb{R}_{+}^{3}}\Lambda_{3}(y,x)\cdot\Theta(y)\mathrm{d}y+\frac{x_{3}}{2\pi}\int_{\{y_{d}=0\}}\frac{\theta_{b}(y_{1})}{|y_{1}-x_{1}|^{2}+|y_{2}-x_{2}|^{2}+x_{3}^{2}}\mathrm{d}y_{1}\mathrm{d}y_{2}\label{Th-3D_half}
\end{equation}
in the case where $d=3$ and $x=(x_{1},x_{2},x_{3})$ with $x_{3}>0$. Where $u$ is a vector field on $\mathbb{R}_{+}^{d}$, $\omega=\nabla\wedge u$ and $\phi=\nabla\cdot u$, and $\theta$ is a function on $\mathbb{R}_{+}^{d}$ and $\Theta=\nabla\theta$,
and $\theta_{b}$ is its boundary value.

We are now in a position to formulate a random vortex (and rate-expansion) dynamics for viscous flows corresponding to the
partial differential equations (\ref{AA-2}, \ref{AA-3}, \ref{AA-4}), which are composed of the following stochastic differential
equations (SDEs): 
\[
\textrm{d}X_{t}^{\eta}=\sqrt{2\mu\rho^{-1}(X_{t}^{\eta},t)}\mathrm{d}B_{t}^{X}+(\mu\nabla\rho^{-1}+u)(X_{t}^{\eta},t)\mathrm{d}t,\quad X_{0}^{\eta}=\eta,
\]
\[
\textrm{d}Y_{t}^{\eta}=\sqrt{2\kappa\rho_0^{-1}(Y_{t}^{\eta},t)}\mathrm{d}B_{t}^{Y}+(\kappa\nabla\rho_0^{-1}+u)(Y_{t}^{\eta},t)\mathrm{d}t,\quad Y_{0}^{\eta}=\eta
\]
and 
\[
\textrm{d}Z_{t}^{\eta}=\sqrt{2(\mu+\lambda)\rho^{-1}(Z_{t}^{\eta},t)}\mathrm{d}B_{t}^{Z}+((\mu+\lambda)\nabla\rho^{-1}+u)(Z_{t}^{\eta},t)\mathrm{d}t,\quad Z_{0}^{\eta}=\eta,
\]
where $u(x,t)$ and $\rho(x,t)$ are the flow velocity and the fluid density, $B^{X}$, $B^{Y}$ and $B^{Z}$ are independent
of $d$-dimensional Brownian motions on a probability space $(\varOmega,\mathcal{F},\mathbb{P})$.

Since in this work we consider only the fluid flows with small viscosity and small fluid gradient $\nabla\rho^{-1}$, and
further assume that $\rho_{0}$ is a constant, we instead consider the following simplified (approximate) system of SDEs:
\begin{equation}
\textrm{d}X_{t}^{\eta}=\sqrt{2\mu\rho^{-1}}\mathrm{d}B_{t}^{X}+u(X_{t}^{\eta},t)\mathrm{d}t,\quad X_{0}^{\eta}=\eta,\label{sde X-3}
\end{equation}

\begin{equation}
\textrm{d}Y_{t}^{\eta}=\sqrt{2\kappa\rho_0^{-1}}\mathrm{d}B_{t}^{Y}+u(Y_{t}^{\eta},t)\mathrm{d}t,\quad Y_{0}^{\eta}=\eta\label{sde Y-3}
\end{equation}
and

\begin{equation}
\textrm{d}Z_{t}^{\eta}=\sqrt{2(\mu+\lambda)\rho^{-1}}\mathrm{d}B_{t}^{Z}+u(Z_{t}^{\eta},t)\mathrm{d}t,\quad Z_{0}^{\eta}=\eta\label{sde Z-3}
\end{equation}
for every $\eta\in\mathbb{R}^{d}$. If $u(x,t)$ is considered as given in the previous system, then $X$, $Y$ and $Z$
are independent diffusions, and three SDEs are decoupled.

In general for a given $\tau\geq0$, we will use $X^{\eta,\tau}$, $Y^{\iota,\tau}$ and $Z^{\xi,\tau}$ to denote the diffusions
defined by SDEs 
\begin{equation}
\textrm{d}X_{t}^{\eta,\tau}=\sqrt{2\mu\rho^{-1}}\mathrm{d}B_{t}^{X}+u(X_{t}^{\eta,\tau},t)\mathrm{d}t,\quad X_{s}^{\eta,\tau}=\eta\quad\mathrm{for}\;s\leq\tau,\label{sde X-1-2}
\end{equation}

\begin{equation}
\textrm{d}Y_{t}^{\iota,\tau}=\sqrt{2\kappa\rho_0^{-1}}\mathrm{d}B_{t}^{Y}+u(Y_{t}^{\iota,\tau},t)\mathrm{d}t,\quad Y_{s}^{\iota,\tau}=\iota\quad\mathrm{for}\;s\leq\tau\label{sde Y-1-2}
\end{equation}
and

\begin{equation}
\textrm{d}Z_{t}^{\xi,\tau}=\sqrt{2(\mu+\lambda)\rho^{-1}}\mathrm{d}B_{t}^{Z}+u(Z_{t}^{\xi,\tau},t)\mathrm{d}t,\quad Z_{s}^{\xi,\tau}=\xi\quad\mathrm{for}\;s\leq\tau.\label{sde Z-1-2}
\end{equation}

Let $S$ be the Jacobian matrix of $u$, that is $S=\left(\frac{\partial}{\partial x_{j}}u^{i}\right)$
, and $\bar{S}$ is the transpose of $S$. According to Theorem \ref{thm3.2}, we then introduce $Q=\big(Q_{j}^{i}(\eta,t;s)\big)$
and $R=\big(R_{j}^{i}(\eta,t;s)\big)$ for $0\leq s\leq t$, for each pair $t$ and $\eta\in\mathbb{R}^{d}$, defined by
solving the following differential equations 
\begin{equation}
\frac{\textrm{d}}{\textrm{d}s}Q_{j}^{i}(\eta,t;s)=-Q_{k}^{i}(\eta,t;s)1_{D}(X_{s}^{\eta})S_{j}^{k}(X_{s}^{\eta},s),\quad Q_{j}^{i}(\eta,t;t)=\delta_{ij}\label{eq:dQ=00003D00003D00003DQS1_D-2}
\end{equation}
and 
\begin{equation}
\frac{\textrm{d}}{\textrm{d}s}R_{j}^{i}(\eta,t;s)=R_{k}^{i}(\eta,t;s)1_{D}(Z_{s}^{\eta})\bar{S}_{j}^{k}(Z_{s}^{\eta},s),\quad R_{j}^{i}(\eta,t;t)=\delta_{ij}\label{eq:dR=00003D00003D00003DRS^1_D-2}
\end{equation}
for $t>0$, where $i,j=1,\cdots,d$.

\section{Wall-bounded vortex system}

In this section, we consider the model (\ref{AA-c}, \ref{AA-2}, \ref{AA-3}, \ref{AA-4}, \ref{AA-5}) of Oberbeck-Boussinesq
flows with a heating source in the fluid in $\mathbb{R}_{+}^{d}$, $d=2$ or $3$. In \citep{Z Qian},
there's a discussion about incompressible Oberbeck-Boussinesq flows but without the consideration of variation of density and expansion
of fluid. Since the flow is driven by joint forces of gravity and buoyancy, the discussion of density changes in buoyancy
due to the temperature is inevitable. Thanks to the tool developed in \citep{Z.Qian and Z.Shen}, we are able to deal with
fluid with expansion and density variation. We derive the random vortex and expansion-rate formulation for (compressible)
flows past a solid wall. The underlying flow in $D=\mathbb{R}_{+}^{d}=\left\{ (\boldsymbol{x},x_{d})\in\mathbb{R}^{d}:x_{d}>0\right\} $
is modeled by the fluid dynamics ($d=2$ or $3$), whose boundary $\partial D$ where $x_{d}=0$ is considered as the solid
wall.

The introduction of Oberbeck-Boussinesq equations is presented in Section 2 (\ref{AA-c}-\ref{AA-5}), and the Feynman-Kac
type formula established in Section 3 allows us to work out the functional integral representations for vorticity, expansion
rate and temperature gradient. To simplify the notation, we use $p_{a}(s,x;t,y)$ to denote the transition probability density
function for the diffusion with infinitesimal generator $a\rho^{-1}\Delta+u\cdot\nabla$, where $a=\kappa$, $\mu$ or
$\mu+\lambda$, where $u(x,t)$ is the velocity of the flow in question.

Observe that $\Theta=\nabla\theta$ satisfies the parabolic type equation: 
\begin{equation}
\partial_{t}\Theta+(u\cdot\nabla)\Theta+\bar{S}\Theta=\frac{\kappa}{\rho_{0}}\Delta\Theta+\nabla g\label{AA_4_1}
\end{equation}
where  $(\bar{S}\Theta)^{j}=\bar{S}_{l}^{j}\Theta^{l}$. 

The flow is described by a time-independent vector field, the velocity $u(x,t)$ (for $x=(\boldsymbol{x},x_{d})$ where $x_{d}\geq0$),
satisfying the non-slip condition such that $u(x,t)=0$ for $x\in\partial D$ and $t>0$. The vector field $u(x,t)$ is extended to $\mathbb{R}^{d}$
through reflection method, so that 
\[
u^{i}(\bar{x},t)=u^{i}(x,t)\quad\mathrm{for}\ i=1,...,d-1\ \mathrm{and}\ u^{d}(\bar{x},t)=-u^{d}(x,t)
\]
where $\bar{x}=(\boldsymbol{x},-x_{d}).$ Using the same convention as in the previous sections, $p_{\mu}(\tau,\xi,t,x)$
denotes the transition probability density function of $\mathscr{L}_{\mu}=\mu\Delta+u(\cdot,t)\cdot\nabla$ on $\mathbb{R}^{d}$.
Then it can be verified that the transition probability density function of the $\mathscr{L}_{\mu}$-diffusion killed on
leaving the domain $D$ is given by 
\[
p_{\mu}^{D}(\tau,\xi,t,x)=p_{\mu}(\tau,\xi,t,x)-p_{\mu}(\tau,\bar{\xi},t,x)\quad\textrm{ for }\xi,x\in D.
\]
As a consequence we have 
\[
\mathbb{E}\left[\varphi(X_{t}^{\xi,\tau})1_{\{t<\zeta(X_{t}^{\xi,\tau})\}}\right]=\mathbb{E}\left[\varphi(X_{t}^{\xi,\tau})1_{D}(X_{t}^{\xi,\tau})\right]-\mathbb{E}\left[\varphi(X_{t}^{\bar{\xi},\tau})1_{D}(X_{t}^{\bar{\xi},\tau})\right]
\]
for any bounded Borel measurable $\varphi$. The method elaborated in \citep{Z.Qian 2022} has been used for the simulation
of Navier-Stokes equations with boundary.

We observe that the vorticity, temperature and the temperature gradient possess non-homogeneous Dirichlet boundary conditions, which does not satisfy the requirements of the representation theorem in Section 3. Thus, cut-off method is employed to address this issue. The boundary values of $\omega,\phi,\Theta$ at the wall $\partial D$, denoted by $\omega_{b}(\boldsymbol{x},t)$, $\phi_{b}(\boldsymbol{x},t)$,
$\Theta_{b}(\boldsymbol{x},t)$ respectively, can not be determined a priori in general, and are time-dependent. For simplicity
we use the following notation: if $f(x)$ is a tensor field on $D$, then for every $\varepsilon>0$, $f^{\varepsilon}(x,t)=f(x,t)-f_{b}(\boldsymbol{x},t)\psi(x_{d}/\varepsilon)$
for $x=(\boldsymbol{x},x_{d})\in D$, where $f_{b}(\boldsymbol{x},t)=f((\boldsymbol{x},0),t)$ is the trace of $f$ along
$\partial D$. Here $\psi:[0,\infty)\rightarrow[0,1]$ is a smooth cut-off function such that $\psi(s)=1_{\{0\leq s<1/3\}}+1_{\{1/3\leq s<2/3\}}h(s)$,
where $h(s)$ only need to guarantee the smooth of $\psi$.

By construction, $\lim_{\varepsilon\rightarrow0^{+}}\psi(x_{d}/\varepsilon)=1_{\{x_{d}=0\}}.$ In particular $f^{\varepsilon}$
satisfies the homogeneous Dirichlet boundary conditions.

For every $\varepsilon>0$, according to (\ref{AA-2}, \ref{AA-3} and \ref{AA-5}), 
\begin{equation}
\sum_{i,j}\frac{\partial}{\partial x^{j}}\left(\frac{\mu}{\rho}\delta_{ij}\frac{\partial}{\partial x^{i}}\omega^{\varepsilon}\right)-(u\cdot\nabla)\omega^{\varepsilon}-\phi\omega^{\varepsilon}-\partial_{t}\omega^{\varepsilon}+S\omega^{\varepsilon}+\nabla\wedge F(\theta)+\chi^{\varepsilon}=0,\label{omega mod}
\end{equation}
\begin{equation}
\sum_{i,j}\frac{\partial}{\partial x^{j}}\left(\frac{\mu+\lambda}{\rho}\delta_{ij}\frac{\partial}{\partial x^{i}}\phi^{\varepsilon}\right)-(u\cdot\nabla)\phi^{\varepsilon}-\partial_{t}\phi^{\varepsilon}+\vartheta^{\varepsilon}=0,\label{phi mod}
\end{equation}
and 
\begin{equation}
\frac{\kappa}{\rho_{0}}\Delta\Theta^{\varepsilon}-(u\cdot\nabla)\Theta^{\varepsilon}-\partial_{t}\Theta^{\varepsilon}-\bar{S}\Theta^{\varepsilon}+\nabla g+\alpha^{\varepsilon}=0\label{vartheta mod}
\end{equation}
where $\chi^{\varepsilon}$, $\vartheta^{\varepsilon}$ and $\alpha^{\varepsilon}$ are the error terms, which can be simplified
later on, given as the followings: 
\begin{align}
\chi^{\varepsilon}(x,t) & =\psi(\frac{x_{d}}{\varepsilon})\left[\sum_{i,j}\frac{\partial}{\partial x^{j}}\left(\frac{\mu}{\rho}\delta_{ij}\frac{\partial}{\partial x^{i}}\omega^{\varepsilon}\right)-\sum_{j=1}^{d-1}u^{j}\frac{\partial}{\partial x_{j}}-\phi-\frac{\partial}{\partial t}+S\right]\omega_{b}(\boldsymbol{x},t)\nonumber \\
 & -\frac{1}{\varepsilon}\psi^{\prime}(\frac{x_{d}}{\varepsilon})\omega_{b}(\boldsymbol{x},t)u^{d}(x,t)+\frac{\mu}{\rho}\frac{1}{\varepsilon^{2}}\psi^{\prime\prime}(\frac{x_{d}}{\varepsilon})\omega_{b}(\boldsymbol{x},t),\label{eq:chi}\\
\vartheta^{\varepsilon}(x,t) & =\psi(\frac{x_{d}}{\varepsilon})\left[\sum_{i,j}\frac{\partial}{\partial x^{j}}\left(\frac{\mu+\lambda}{\rho}\delta_{ij}\frac{\partial}{\partial x^{i}}\phi^{\varepsilon}\right)-\sum_{j=1}^{d-1}u^{j}\frac{\partial}{\partial x_{j}}-\frac{\partial}{\partial t}\right]\phi_{b}(\boldsymbol{x},t)\nonumber \\
 & -\frac{1}{\varepsilon}\psi^{\prime}(\frac{x_{d}}{\varepsilon})\phi_{b}(\boldsymbol{x},t)u^{d}(x,t)+\frac{\mu+\lambda}{\rho}\frac{1}{\varepsilon^{2}}\psi^{\prime\prime}(\frac{x_{d}}{\varepsilon})\phi_{b}(\boldsymbol{x},t), \label{eq:vartheta}\\
\alpha^{\varepsilon}(x,t) & =\psi(\frac{x_{d}}{\varepsilon})\left[\frac{\kappa}{\rho_0}\Delta-\sum_{j=1}^{d-1}u^{j}\frac{\partial}{\partial x_{j}}-\frac{\partial}{\partial t}-\bar{S}\right]\Theta_{b}(\boldsymbol{x},t)\nonumber \\
 & -\frac{1}{\varepsilon}\psi^{\prime}(\frac{x_{d}}{\varepsilon})\Theta_{b}(\boldsymbol{x},t)u^{d}(x,t)+\frac{\kappa}{\rho_{0}}\frac{1}{\varepsilon^{2}}\psi^{\prime\prime}(\frac{x_{d}}{\varepsilon})\Theta_{b}(\boldsymbol{x},t)\label{eq:alpha}
\end{align}
where $S=(S_{l}^{i})$ with $S_{l}^{i}=\frac{\partial u^{i}}{\partial x_{l}}$ and $\bar{S}$ is the transpose matrix. We
can see that for the first terms on the right hand side in each equation, by taking the limit of $\varepsilon$ to $0$,
$\psi(x_{d}/\varepsilon)$ goes to $1_{\{x_{d}=0\}}$. Hence we can eliminate the first terms for $\chi^{\varepsilon}$,
$\vartheta^{\varepsilon}$ and $\alpha^{\varepsilon}$ if we put them into a Lebesgue integration. What we need to do is
to take care of the terms involved with $-\varepsilon^{-1}\psi^{\prime}(x_{d}/\varepsilon)$ and $\varepsilon^{-2}\psi^{\prime\prime}(x_{d}/\varepsilon)$.

Obviously, $\omega^{\varepsilon},\phi^{\varepsilon},\Theta^{\varepsilon}$ all satisfy the homogeneous Dirichlet boundary conditions. Applying the integral representation Theorem \ref{thm3.2} to $\omega^{\varepsilon},\phi^{\varepsilon},\Theta^{\varepsilon}$,
we are able to establish the following representations. 
\begin{thm}
\label{prop:theta mod} For every $\varepsilon>0$, $\omega,\phi,\Theta$ have the following representations: 
\begin{align}
\Theta(\xi,t) & =\psi\left(\frac{\xi_{d}}{\varepsilon}\right)\Theta_{b}(\boldsymbol{\xi},t)+\int_{\mathbb{R}_{+}^{d}}\mathbb{E}\left[\left.1_{\{t<\zeta(Y^{\eta})\}} \textrm{e}^{\int_{0}^{t}\phi(Y_{r},r)\textrm{d}r} R(\eta,t;0)\Theta^{\varepsilon}(\eta,0)\right|Y_{t}^{\eta}=\xi\right]p_{\kappa}(0,\eta;t,\xi)\textrm{d}\eta\nonumber \\
 & +\int_{0}^{t}\int_{\mathbb{R}_{+}^{d}}\mathbb{E}\left[\left.1_{\{s>\gamma_{t}(Y^{\eta})\}}\textrm{e}^{\int_{0}^{t}\phi(Y_{r},r)\textrm{d}r}R(\eta,t;s)\left(\nabla g(Y_{s}^{\eta},s)+\alpha^{\varepsilon}(Y_{s}^{\eta},s)\right)\right|Y_{t}^{\eta}=\xi\right]\label{flower theta bounded}\\
 & \ \times p_{\kappa}(0,\eta;t,\xi)\textrm{d}\eta\textrm{d}s,\nonumber 
\end{align}
\begin{align}
\omega(\xi,t) & =\psi\left(\frac{\xi_{d}}{\varepsilon}\right)\omega_{b}(\boldsymbol{\xi},t)+\int_{\mathbb{R}_{+}^{d}}\mathbb{E}\left[\left.1_{\{t<\zeta(X^{\eta})\}}Q(\eta,t;0)\omega^{\varepsilon}(\eta,0)\right|X_{t}^{\eta}=\xi\right]p_{\mu}(0,\eta;t,\xi)\textrm{d}\eta\nonumber \\
 & +\int_{0}^{t}\int_{\mathbb{R}_{+}^{d}}\mathbb{E}\left[\left.1_{\{s>\gamma_{t}(X^{\eta})\}}\textrm{e}^{\int_{0}^{s}\phi(X_{r},r)\textrm{d}r}Q(\eta,t;s)\left(\nabla\wedge F(X_{s}^{\eta},s)+\chi^{\varepsilon}(X_{s}^{\eta},s)\right)\right|X_{t}^{\eta}=\xi\right]\label{omega bounded}\\
\  & \times p_{\mu}(0,\eta;t,\xi)\textrm{d}\eta\textrm{d}s,\nonumber 
\end{align}
and 
\begin{align}
\phi(\xi,t) & =\psi\left(\frac{\xi_{d}}{\varepsilon}\right)\phi_{b}(\boldsymbol{\xi},t)+\int_{\mathbb{R}_{+}^{d}}\mathbb{E}\left[\left.1_{\{t<\zeta(Z^{\eta})\}}\textrm{e}^{\int_{0}^{t}\phi(Z_{r},r)\textrm{d}r}\phi^{\varepsilon}(\eta,0)\right|Z_{t}^{\eta}=\xi\right]p_{\mu+\lambda}(0,\eta;t,\xi)\textrm{d}\eta\label{phi bounded}\\
 & +\int_{0}^{t}\int_{\mathbb{R}_{+}^{d}}\mathbb{E}\left[\left.1_{\{s>\gamma_{t}(Z^{\eta})\}}\textrm{e}^{\int_{0}^{t}\phi(Z_{r},r)\textrm{d}r}\vartheta^{\varepsilon}(Z_{s}^{\eta},s)\right|Z_{t}^{\eta}=\xi\right]p_{\mu+\lambda}(0,\eta;t,\xi)\textrm{d}\eta\textrm{d}s\nonumber 
\end{align}
for every $\xi=(\boldsymbol{\xi},\xi_{d})\in\mathbb{R}_{+}^{d},d=2$ or $3$ and $t>0$, where 
\[
\frac{\textrm{d}}{\textrm{d}t}R_{j}^{i}(\psi,T;t)=-R_{k}^{i}(\psi,T;t)1_{D}(Y_{s}^{\eta})\bar{S}_{j}^{k}(\psi(t),t)
\]
\[
R_{j}^{i}(\psi,T;T)=\delta_{ij}
\]
and 
\[
\frac{\textrm{d}}{\textrm{d}t}Q_{j}^{i}(\psi,T;t)=-Q_{k}^{i}(\psi,T;t)1_{D}(X_{s}^{\eta})S_{j}^{k}(\psi(t),t)
\]
and 
\[
Q_{j}^{i}(\psi,T;T)=\delta_{ij}.
\]
\end{thm}

\begin{proof}
Take $\omega$ as an example, we apply Theorem \ref{thm3.2} to (\ref{omega mod}), we have 
\begin{align*}
\omega^{\varepsilon}(\xi,t) & =\int_{\mathbb{R}^{3}}\mathbb{E}\left[\left.1_{\{t<\zeta(X^{\eta})\}}\textrm{e}^{\int_{0}^{t}\phi(X_{r},r)\textrm{d}r}A(\eta,t;0)\omega^{\varepsilon}(\eta,0)\right|X_{t}^{\eta}=\xi\right]p_{\mu}(0,\eta;t,\xi)\textrm{d}\eta\\
 & +\int_{0}^{t}\int_{\mathbb{R}^{3}}\mathbb{E}\left[\left.1_{\{s>\gamma_{t}(X^{\eta})\}}\textrm{e}^{\int_{0}^{t}\phi(X_{r},r)\textrm{d}r}A(\eta,t;s)\left(\nabla\wedge F(X_{s}^{\eta},s)+\chi^{\varepsilon}\right)\right|X_{t}^{\eta}=\xi\right]p_{\mu}(0,\eta;t,\xi)\textrm{d}\eta\textrm{d}s
\end{align*}
where 
\[
\frac{\textrm{d}}{\textrm{d}s}A_{j}^{i}(\eta,t;s)=-A_{k}^{i}(\eta,t;s)S_{j}^{k}(X_{s}^{\eta},s)+A_{j}^{i}(\eta,t;s)\phi(X_{s}^{\eta},s),\quad A_{j}^{i}(\eta,t;t)=\delta_{ij}.
\]
Since $Q_{j}^{i}(\eta,t;s)=e^{\int_{s}^{t}\phi(X_{s}^{\eta},s)\mathrm{d}s}A_{j}^{i}(\eta,t;s)$ for $t\in[0,T],$ the conclusion
follows immediately. Now since $\omega^{\varepsilon}(x,t)=\omega(x,t)-\omega_{b}(\boldsymbol{x},t)\psi(x_{d}/\varepsilon)$,
we have the desired formula. Formulas \ref{flower theta bounded} and \ref{phi bounded} comes alike.
\end{proof}
Combining these representations with Biot-Savart laws, we are able to establish integral functional representations for the
velocity and the temperature. 
\begin{thm}
For the temperature \textup{$\theta$}, the following representations holds: 
\begin{enumerate}
\item When $d=2$, 
\begin{align*}
\theta(x,t) & =\frac{x_{2}}{\pi}\int_{-\infty}^{\infty}\theta_{b}(\xi_{1},t)\frac{1}{\left|\xi_{1}-x_{1}\right|^{2}+x_{2}^{2}}d\xi_{1}\\
 & -\int_{0}^{\varepsilon}\left[\psi(\frac{\xi_{2}}{\varepsilon})\int_{\mathbb{R}}\Lambda_{2}(\xi,x)\cdot\Theta_{b}(\xi_{1},t)d\xi_{1}\right]d\xi_{2}\\
 & -\int_{\mathbb{R}_{+}^{2}}\mathbb{E}\left[1_{\{t<\zeta(Y^{\eta})\}}1_{\mathbb{R}_{+}^{2}}\textrm{e}^{\int_{0}^{t}\phi(Y_{r},r)\textrm{d}r}(Y_{t}^{\eta})\Lambda_{2}(Y_{t}^{\eta},x)\cdot R(\eta,t;0)\Theta^{\varepsilon}(\eta,0)\right]d\eta\\
 & -\int_{0}^{t}\int_{\mathbb{R}_{+}^{2}}\mathbb{E}\Big[1_{\{s>\gamma_{t}(Y^{\eta})\}}1_{\mathbb{R}_{+}^{2}}(Y_{t}^{\eta})\Lambda_{2}(Y_{t}^{\eta},x)\cdot\textrm{e}^{\int_{0}^{t}\phi(Y_{r},r)\textrm{d}r}R(\eta,t;s)\\
 & \ \times\left(\nabla g(Y_{s}^{\eta},s)+\alpha^{\varepsilon}(Y_{s}^{\eta},s)\right)\Big]\textrm{d}\eta\textrm{d}s
\end{align*}
for every $x\in\mathbb{R}_{+}^{2}$ and $t>0$. 
\item When $d=3$, 
\begin{align*}
\theta(x,t) & =\frac{x_{3}}{2\pi}\int_{\mathbb{R}^{2}}\theta_{b}(\xi_{1},\xi_{2},t)\frac{1}{\left|\xi_{1}-x_{1}\right|^{2}+\left|\xi_{2}-x_{2}\right|^{2}+x_{3}^{2}}d\xi_{1}d\xi_{2}\\
 & -\int_{0}^{\varepsilon}\left[\psi(\frac{\xi_{3}}{\varepsilon})\int_{\mathbb{R}^{2}}\Lambda_{3}(\xi,x)\cdot\Theta_{b}(\xi_{1},\xi_{2},t)d\xi_{1}d\xi_{2}\right]d\xi_{3}\\
 & -\int_{\mathbb{R}_{+}^{3}}\mathbb{E}\left[1_{\{t<\zeta(Y^{\eta})\}}1_{\mathbb{R}_{+}^{3}}\textrm{e}^{\int_{0}^{t}\phi(Y_{r},r)\textrm{d}r}(Y_{t}^{\eta})\Lambda_{3}(Y_{t}^{\eta},x)\cdot R(\eta,t;0)\Theta^{\varepsilon}(\eta,0)\right]d\eta\\
 & -\int_{0}^{t}\int_{\mathbb{R}_{+}^{3}}\mathbb{E}\Big[1_{\{s>\gamma_{t}(Y^{\eta})\}}1_{\mathbb{R}_{+}^{3}}(Y_{t}^{\eta})\Lambda_{3}(Y_{t}^{\eta},x)\cdot\textrm{e}^{\int_{0}^{t}\phi(Y_{r},r)\textrm{d}r}R(\eta,t;s)\\
 & \ \times\left(\nabla g(Y_{s}^{\eta},s)+\alpha^{\varepsilon}(Y_{s}^{\eta},s)\right)\Big]\textrm{d}\eta\textrm{d}s
\end{align*}
for every $x\in\mathbb{R}_{+}^{3}$ and $t>0$. 
\end{enumerate}
\end{thm}

\begin{thm}
\label{thm6.6} For the velocity $u$, the following representations holds: 
\begin{enumerate}
\item When $d=2$, 
\begin{align*}
u(x,t) & =\int_{0}^{\varepsilon}\left[\psi\left(\frac{\xi_{2}}{\varepsilon}\right)\int_{-\infty}^{\infty}\Lambda_{2}(\xi,x)\wedge\omega_{b}(\xi_{1},t)-\Lambda_{2}(\xi,x)\phi_{b}(\xi_{1},t)d\xi_{1}\right]d\xi_{2}\\
 & +\int_{\mathbb{R}_{+}^{2}}\mathbb{E}\left[1_{\{t<\zeta(X^{\eta})\}}1_{\mathbb{R}_{+}^{2}}(X_{t}^{\eta})\Lambda_{2}(X_{t}^{\eta},x)\wedge\omega^{\varepsilon}(\eta,0)\right]\textrm{d}\eta\\
 & +\int_{0}^{t}\int_{\mathbb{R}_{+}^{2}}\mathbb{E}\left[1_{\{s>\gamma_{t}(X^{\eta})\}}1_{\mathbb{R}_{+}^{2}}(X_{t}^{\eta})\Lambda_{2}(X_{t}^{\eta},x)\wedge\textrm{e}^{\int_{0}^{s}\phi(X_{r},r)\textrm{d}r}\left(\nabla\wedge F(X_{s}^{\eta},s)+\chi^{\varepsilon}(X_{s}^{\eta},s)\right)\right]\textrm{d}\eta\textrm{d}s\\
 & -\int_{\mathbb{R}_{+}^{2}}\mathbb{E}\left[1_{\{t<\zeta(Z^{\eta})\}}1_{\mathbb{R}_{+}^{2}}(Z_{t}^{\eta})\Lambda_{2}(Z_{t}^{\eta},x)\textrm{e}^{\int_{0}^{t}\phi(Z_{r},r)\textrm{d}r}\phi^{\varepsilon}(\eta,0)\right]d\eta\\
 & -\int_{0}^{t}\int_{\mathbb{R}_{+}^{d}}\mathbb{E}\left[1_{\{s>\gamma_{t}(Z^{\eta})\}}1_{\mathbb{R}_{+}^{2}}(Z_{t}^{\eta})\Lambda_{2}(Z_{t}^{\eta},x)\textrm{e}^{\int_{0}^{t}\phi(Z_{r},r)\textrm{d}r}\vartheta^{\varepsilon}(Z_{s}^{\eta},s)\right]\textrm{d}\eta\textrm{d}s
\end{align*}
for every $x\in\mathbb{R}_{+}^{2}$ and $t>0$. 
\item When $d=3$, 
\begin{align*}
u(x,t) & =\int_{0}^{\varepsilon}\left[\psi\left(\frac{\xi_{3}}{\varepsilon}\right)\int_{\mathbb{R}^{2}}\Lambda_{3}(\xi,x)\wedge\omega_{b}(\xi_{1},\xi_{2},t)-\Lambda_{3}(\xi,x)\phi_{b}(\xi_{1},\xi_{2},t)d\xi_{1}d\xi_{2}\right]d\xi_{3}\\
 & +\int_{\mathbb{R}_{+}^{3}}\mathbb{E}\left[1_{\{t<\zeta(X^{\eta})\}}1_{\mathbb{R}_{+}^{2}}(X_{t}^{\eta})\Lambda_{3}(X_{t}^{\eta},x)\wedge Q(\eta,t;0)\omega^{\varepsilon}(\eta,0)\right]\textrm{d}\eta\\
 & +\int_{0}^{t}\int_{\mathbb{R}_{+}^{3}}\mathbb{E}\Big[1_{\{s>\gamma_{t}(X^{\eta})\}}1_{\mathbb{R}_{+}^{3}}(X_{t}^{\eta})\Lambda_{3}(X_{t}^{\eta},x)\wedge\textrm{e}^{\int_{0}^{s}\phi(X_{r},r)\textrm{d}r}Q(\eta,t;s)\\
 & \ \times\left(\nabla\wedge F(X_{s}^{\eta},s)+\chi^{\varepsilon}(X_{s}^{\eta},s)\right)\Big]\textrm{d}\eta\textrm{d}s\\
 & -\int_{\mathbb{R}_{+}^{3}}\mathbb{E}\left[1_{\{t<\zeta(Z^{\eta})\}}1_{\mathbb{R}_{+}^{3}}(Z_{t}^{\eta})\Lambda_{3}(Z_{t}^{\eta},x)\textrm{e}^{\int_{0}^{t}\phi(Z_{r},r)\textrm{d}r}\phi^{\varepsilon}(\eta,0)\right]d\eta\\
 & -\int_{0}^{t}\int_{\mathbb{R}_{+}^{3}}\mathbb{E}\left[1_{\{s>\gamma_{t}(Z^{\eta})\}}1_{\mathbb{R}_{+}^{3}}(Z_{t}^{\eta})\Lambda_{3}(Z_{t}^{\eta},x)\textrm{e}^{\int_{0}^{t}\phi(Z_{r},r)\textrm{d}r}\vartheta^{\varepsilon}(Z_{s}^{\eta},s)\right]\textrm{d}\eta\textrm{d}s
\end{align*}
for every $x\in\mathbb{R}_{+}^{3}$ and $t>0$. 
\end{enumerate}
\end{thm}

\section{Numerical experiments of Oberbeck-Boussinesq flows in \texorpdfstring{$\mathbb{R}_{+}^{2}$}{.}}

In this section, we present several numerical experiments based on simplified equations of compressible Oberbeck-Boussinesq flows. All simulations are grounded in the formulation introduced in Section 5. Owing to the explicit form of the formulation, we can apply a direct discrete scheme to handle the integrals in the representation formulas (converting integrals into finite sums) and use the finite difference method to compute the density. 

Let $\theta_{b}$ denotes the boundary temperature, representing the heat source at the boundary. For 2-dimensional model, the external
force $F$ is given by 
\[
F^{1}(x,t)=0,F^{2}(x,t)=\alpha(\theta(x,t)-\theta_{b}(x_{1},t))
\]
and 
\[
\nabla\wedge F(x,t)=\alpha(\Theta^{1}(x,t)-\frac{\partial}{\partial x_{1}}\theta_{b}(x_{1},t)).
\]

While in 3-dimensional model, the force $F$ applying to the flow has components 
\[
F^{1}(x,t)=0,\quad F^{2}(x,t)=0,\quad F^{3}(x,t)=\alpha(\theta(x,t)-\theta_{b}(x_{1},x_{2},t)).
\]
Therefore 
\begin{align*}
[\nabla\wedge F(x,t)]^{1} & =\frac{\partial}{\partial x_{2}}\theta(x,t)-\frac{\partial}{\partial x_{2}}\theta_{b}(x_{1},x_{2},t),\\{}
[\nabla\wedge F(x,t)]^{2} & =-\frac{\partial}{\partial x_{1}}\theta(x,t)-\frac{\partial}{\partial x_{1}}\theta_{b}(x_{1},x_{2},t),
\end{align*}
and $[\nabla\wedge F(x,t)]^{3}=0$. In numerical experiments, $\varepsilon>0$ is chosen sufficiently small, so that the
error terms $\alpha^{\varepsilon}(x,t),\,\chi^{\varepsilon}(x,t)$and $\vartheta^{\varepsilon}(x,t)$
appearing (\ref{eq:alpha}, \ref{eq:chi}, \ref{eq:vartheta}), according to Theorem \ref{thmA.1} and Theorem
\ref{thmA.2}, can be simplified as follows: 
\[
\alpha^{\varepsilon}(x,t)=\frac{\kappa}{\rho_{0}}\frac{1}{\varepsilon^{2}}\psi^{\prime\prime}(\frac{x_{d}}{\varepsilon})\Theta_{b}(\boldsymbol{x},t),
\]
\[
\chi^{\varepsilon}(x,t)=\frac{\mu}{\rho}\frac{1}{\varepsilon^{2}}\psi^{\prime\prime}(\frac{x_{d}}{\varepsilon})\omega_{b}(\boldsymbol{x},t),
\]
\[
\vartheta^{\varepsilon}(x,t)=\frac{\mu+\lambda}{\rho}\frac{1}{\varepsilon^{2}}\psi^{\prime\prime}(\frac{x_{d}}{\varepsilon})\phi_{b}(\boldsymbol{x},t).
\]

The proof of the previous approximation is quite technical and is left for the Appendix below. In this section we shall report
two numerical experiments, one for an unbounded Oberbeck-Boussinesq flow and another for a wall-bounded flow.

For convenience of the reader, we will restate some notations here. In numerical experiments, we focus on the following
dynamical variables: the velocity $u(x,t)$, the vorticity $\omega(x,t)$, the expansion-rate $\phi(x,t)$, the temperature
$\theta(x,t)$, the density $\rho(x,t)$, and the gradient $\Theta(x,t)$ of $\theta(x,t)$. $\mu>0,\lambda>0$ and $\kappa>0$
are constants related to the physical properties of the fluid. Initial conditions and boundary conditions are specified in
numerical experiments. The following notations will be used: $x^{i_{1},i_{2}}=(i_{1}s,i_{2}s)$, where $s>0$ also denotes
the lattice mesh in this section, $\omega^{i_{1},i_{2}}=\omega(x^{i_{1},i_{1}},0)$, and similarly for $\phi^{i_{1},i_{2}},\Theta^{i_{1},i_{2}}$
for $i_{1},i_{2}\in\mathbb{Z}$.

%\subsection{Oberbeck-Boussinesq flows in \texorpdfstring{$\mathbb{R}_{+}^{2}$}{.}}

Let $\varepsilon>0$ and $\epsilon_{B}>0$ be two small parameters. Since the Biot-Savart kernel is singular, we substitute
$\Lambda_{2}$ by $\Lambda_{2,\epsilon_{B}}$, where

\begin{align*}
\Lambda_{2,\epsilon_{B}}(y,x) & =\left(1-\textrm{e}^{-|y-x|^{2}/\epsilon_{B}}\right)\Lambda_{2}(y,x) \\
&=\left(1-\textrm{e}^{-|y -x|^{2}/\epsilon_{B}}\right)\nabla_y \left( \Gamma_{2}(y,x)-\Gamma_2(y,\bar{x})\right)
\end{align*}

The numerical scheme in this subsection includes $\omega^{\varepsilon}$, $\phi^{\varepsilon}$and $\Theta^{\varepsilon}$
so we define $\omega^{\varepsilon,i_{1},i_{2}}=\omega^{\varepsilon}(x^{i_{1},i_{2}},0)$, $\phi^{\varepsilon,i_{1},i_{2}}=\phi^{\varepsilon}(x^{i_{1},i_{2}},0)$,
$\Theta^{\varepsilon,i_{1},i_{2}}=\Theta^{\varepsilon}(x^{i_{1},i_{2}},0)$. According to non-slip condition (i.e. $u(x,t)=0$
in $\partial D$) of Oberbeck-Boussinesq equations, so that in our model, $\omega^{\varepsilon,i_{1},i_{2}}=\omega(x^{i_{1},i_{2}},0)$,
$\phi^{\varepsilon,i_{1},i_{2}}=\phi(x^{i_{1},i_{2}},0)$, and by the definition in \ref{eq:chi} and \ref{eq:vartheta}.
As the cut-off method mentioned in Section 5, we may choose the concrete cut-off function $\psi$ as follows:

\[
\psi(r)=\begin{cases}
1 & r\in[0,\frac{1}{3})\\
54(r-\frac{1}{2})^{3}-\frac{9}{2}(r-\frac{1}{2})+\frac{1}{2} & r\in[\frac{1}{3},\frac{2}{3}]\\
0 & r\in(\frac{2}{3},\infty)
\end{cases}
\]
which is a $C^{2}$-function. The numerical scheme of Oberbeck-Boussinesq flows in $\mathbb{R}_{+}^{2}$ is introduced as
before.

We carry out numerical experiment based on the one-copy scheme for the upper half-plane case $\mathbb{R}_{+}^{2}$, where
the heat source arises from boundary.  Set time step
$\delta=0.1$, $T=120$. Let the Prandtl numbers($Pr$) of the fluid be 0.15, $\mu=0.15,\kappa=1,\lambda=0.15$,
$\alpha=0.005$, the Biot-Savart law kernel constant $\epsilon_{B}=0.05$, $L=2\pi$ with mesh size $s=0.05\pi$, which is used for integral discretization, then the observation region is $[-L,L]\times[0,L]$. For seeing the flow within the boundary layer
$[-L,L]\times[0,s]$, the thin layer is further partitioned with mesh size $s_{b}=0.1s$ in vertical direction. It should be noted that we set the upper limit and and lower limit of the first integral term in the temperature formula as 100 and -100 respectively to approximate the original improper integral . It is reasonable because the integration value outside of this integral area is so small that can be ignored.

We simulate the scenario of heating a fluid from the boundary. In this case, we set that the fluid is initially in a stationary
state, with a same initial temperature throughout. That is $u(x,0)=(0,0)$, $\omega(x,0)=0$, $\phi(x,0)=0$ and $\theta(x,0)=0.01$.
Then we assume that the external heat source form the boundary is given by 
\[
\theta_{b}(x_{1},t)=\begin{cases}
0.25t\textrm{e}^{-0.5x_{1}^{2}L^{-2}} & t<80\\
20e^{-0.5x_{1}^{2}L^{-2}} & t\geq80
\end{cases}
\]

That is, the external heat source gradually increases, and when the central temperature reaches 20 , the heat source is maintained at that temperature. Hence, 
\[
\Theta_{b}^{1}(x_{1},t)=\begin{cases}
-0.25L^{-2}tx_{1}e^{-0.5x_{1}^{2}L^{-2}} & t<80\\
-20L^{-2}x_{1}e^{-0.5x_{1}^{2}L^{-2}} & t\geq80
\end{cases}
\]
and we set the heating source $g(x,t)=(\theta(x,0)+0.05t)\textrm{e}^{-0.5|x|^{2}L^{-2}}$, so that 
\[
\nabla g(x,t)=\left[-(0.01+0.05t)x_{1}L^{-2}e^{-0.5x^{2}L^{-2}};-(0.01+0.05t)x_{2}L^{-2}e^{-0.5x^{2}L^{-2}}\right].
\]

\begin{figure}[H]
\begin{centering}
\includegraphics[width=1\textwidth]{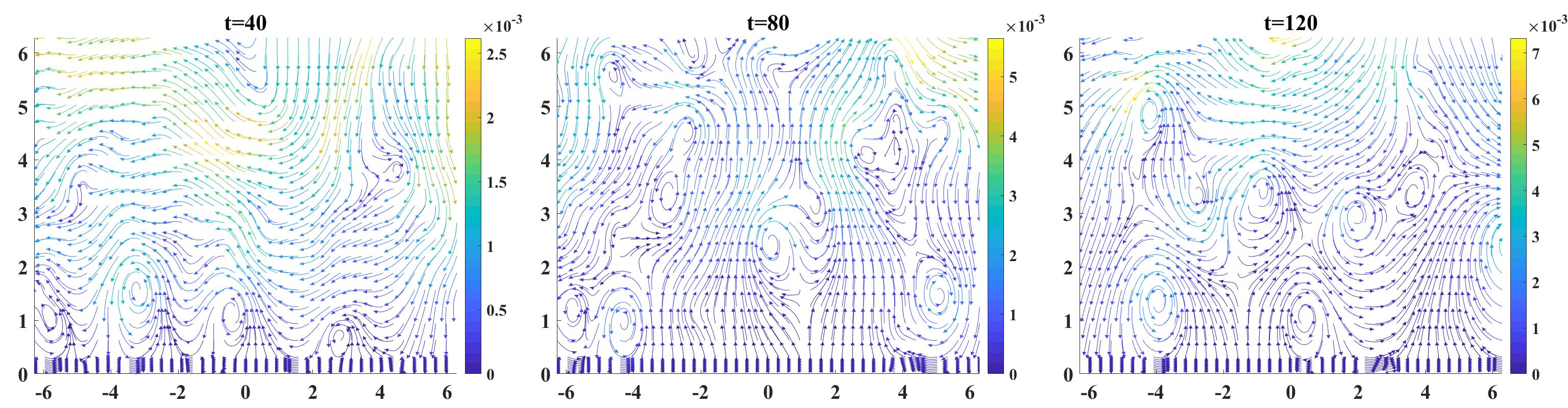}
\par\end{centering}
\caption{Velocity fields of Oberbeck-Boussinesq flows on $\mathbb{R}_{+}^{2}$}\label{fig:bounded u} 
\end{figure}
Figure \ref{fig:bounded u} shows that, by applying the heating source at the boundary ($x_2=0$), rotating vector
fields are generated from the boundary (Figure \ref{fig:bounded u}, t=40), and more and more vortices are formed afterwards, which 
make the vortex system more chaotic and gradually diffuse into other regions of the fluid region (see Figure \ref{fig:bounded u}, t=80 and 120). With the continuous addition
of heat, the temperature (cf. Figure \ref{fig:bounded theta})  keeps rising and the velocity of the fluid flow increases progressively. 
\begin{figure}[H]
\begin{centering}
\includegraphics[width=1\textwidth]{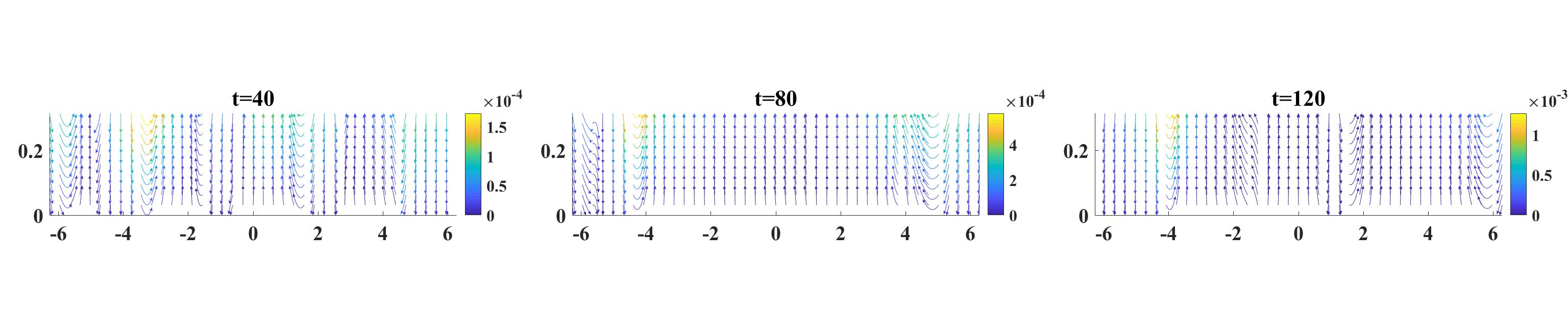}\par\end{centering}
\caption{Boundary velocity fields of Oberbeck-Boussinesq flows on $\mathbb{R}_{+}^{2}$}\label{fig:bounded u_b} 
\end{figure}

To better demonstrate the behavior of the flow within the boundary layer, we have used a refined partitioning at the boundary of the fluid system, shown in Figure \ref{fig:bounded u_b}. B\'enard convection as well as the hairy
type of boundary flows are presented in Figure \ref{fig:bounded u} and \ref{fig:bounded u_b}, confirming
the theoretical models and observations.

\begin{figure}[H]
\begin{centering}
\includegraphics[width=1\textwidth]{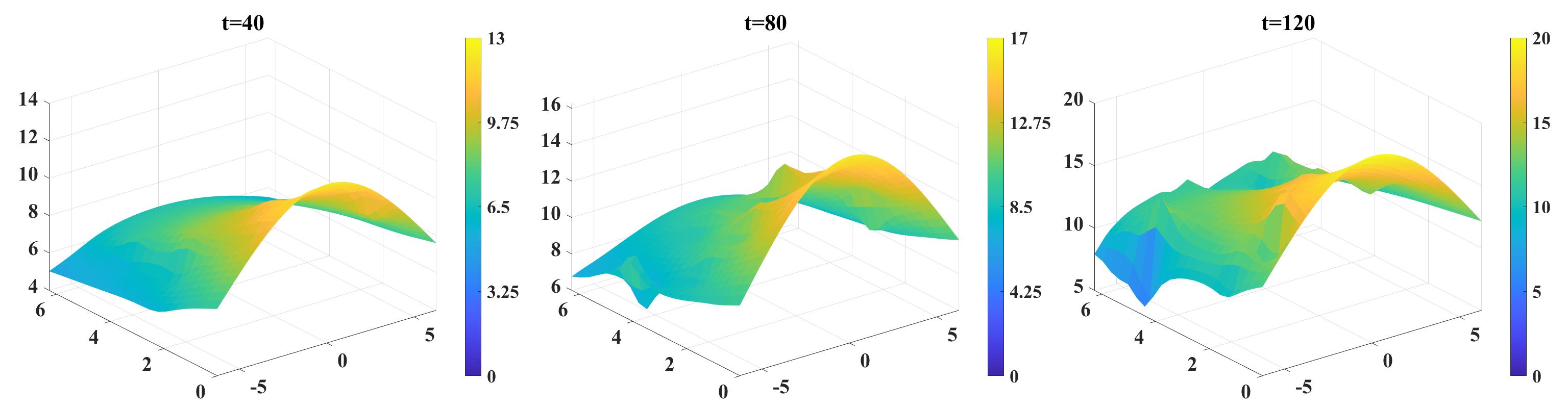}\par\end{centering}
\caption{Temperature of Oberbeck-Boussinesq flows on $\mathbb{R}_{+}^{2}$}\label{fig:bounded theta} 
\end{figure}
Because we add the boundary heating source in the form of a compactly supported
function, the Figure \ref{fig:bounded theta} clearly shows that the distribution of temperature
conforms to the solution of the heat equation with some fluctuations caused by motions of Brownian particles. The temperature increases as heat continues to be added. 
\begin{figure}[H]
\begin{centering}
\includegraphics[width=1\textwidth]{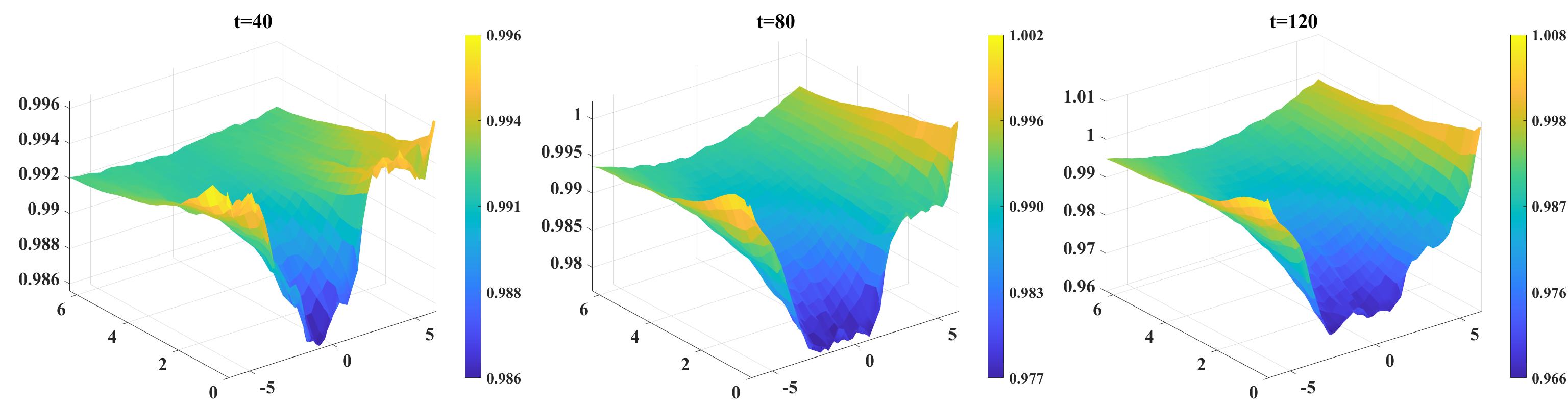}\par\end{centering}
\caption{Density of Oberbeck-Boussinesq flows on $\mathbb{R}_{+}^{2}$}\label{fig:bounded density}
\end{figure}

Although the variation of the fluid density is not so big, the density still changes as demonstrating in Figure \ref{fig:bounded density}, which is consistent with our assumption of near incompressibility. By comparing Figures \ref{fig:bounded theta} with Figure \ref{fig:bounded density}, we observe that the distribution of density is essentially inverse to that of temperature, with density decreasing in regions of higher temperature, which is consistent with the existing theories.

\begin{figure}[h]
\begin{centering}
\includegraphics[width=1\textwidth]{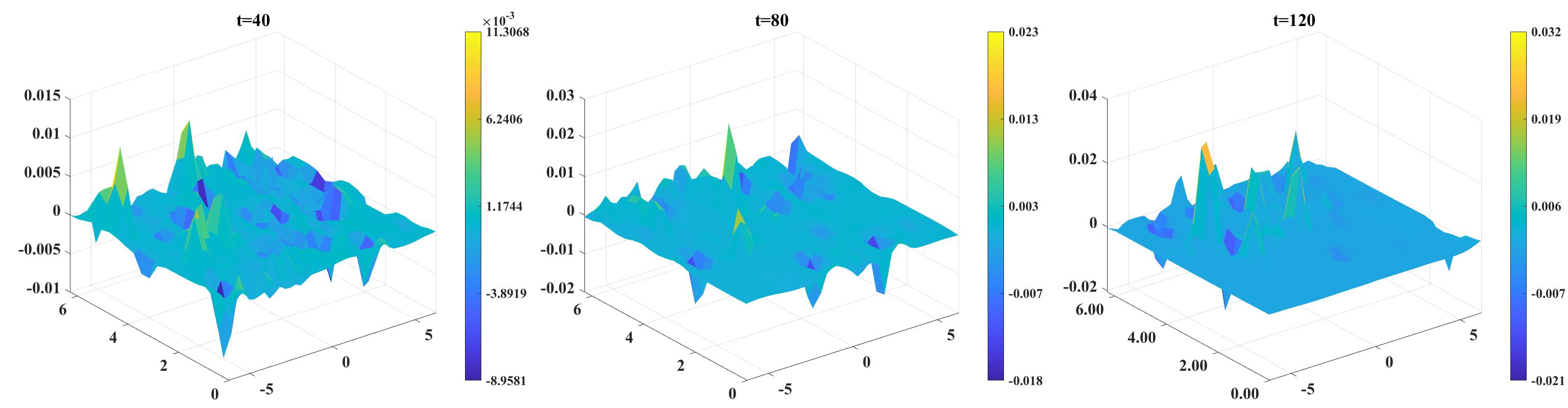}\par\end{centering}
\caption{Vorticity of Oberbeck-Boussinesq flows on $\mathbb{R}_{+}^{2}$}\label{fig:bounded w}
\end{figure}

\begin{figure}[H]
\noindent \begin{centering}
\includegraphics[width=1\textwidth]{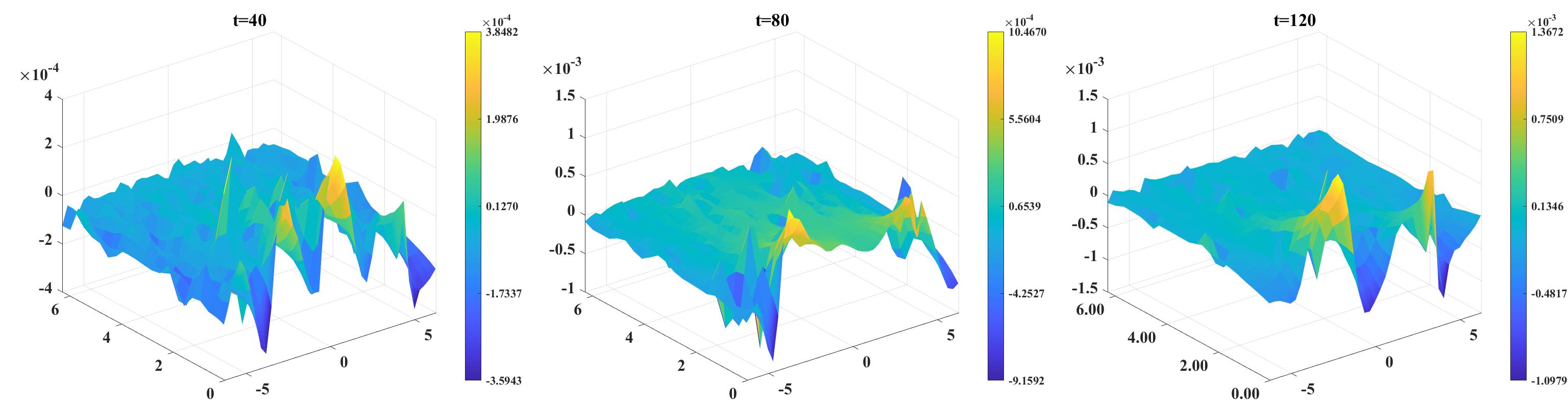}\par\end{centering}
\caption{Expansion rate of Oberbeck-Boussinesq flows on $\mathbb{R}_{+}^{2}$}\label{fig:bounded phi}
\end{figure}

Figure \ref{fig:bounded density}, \ref{fig:bounded w} and \ref{fig:bounded phi} show that the changes in density, vorticity, and expansion rate are all not so big, which aligns with our hypothesis. It is evident that the expansion rate and vorticity exhibit some fluctuations, however, the  expansion rate is significantly smaller than that of other physical quantities. All these experimental results are consistent with both theoretical expectations and empirical observations, indicating the validity of our theoretical model and numerical scheme. 

\section*{Appendix - Estimating error terms}

In this part, we will show the consistency of approximation for error terms $\alpha^{\varepsilon}(x,t)$, $\chi^{\varepsilon}(x,t)$
and $\vartheta^{\varepsilon}(x,t)$ in previous sections, such that
\begin{align*}
\lim_{\epsilon\rightarrow0}\chi^{\varepsilon}(x,t) & =\lim_{\epsilon\rightarrow0}\frac{\mu}{\rho}\frac{1}{\varepsilon^{2}}\psi^{\prime\prime}(\frac{x_{d}}{\varepsilon})\omega_{b}(\boldsymbol{x},t),\\
\lim_{\epsilon\rightarrow0}\alpha^{\varepsilon}(x,t) & =\lim_{\epsilon\rightarrow0}\frac{\kappa}{\rho_{0}}\frac{1}{\varepsilon^{2}}\psi^{\prime\prime}(\frac{x_{d}}{\varepsilon})\Theta_{b}(\boldsymbol{x},t),\\
\lim_{\epsilon\rightarrow0}\vartheta^{\varepsilon}(x,t) & =\lim_{\epsilon\rightarrow0}\frac{\mu+\lambda}{\rho}\frac{1}{\varepsilon^{2}}\psi^{\prime\prime}(\frac{x_{d}}{\varepsilon})\phi_{b}(\boldsymbol{x},t).
\end{align*}
The notation $p_{\lambda\rho^{-1},u}$ is simplified to be $p_{\lambda}$ and the velocity is extended for all $x\in\mathbb{R}^{d}=D$,
so that 
\[
p_{\lambda}^{D}(s,\xi,t,x)=p_{\lambda}(\tau,\xi,t,x)-p_{\lambda}(\tau,\bar{\xi},t,x)
\]
is clearly the Green function to Dirichlet problem of $\mathscr{L}_{\lambda\rho^{-1},-u}-\frac{\partial}{\partial t}$. We
will again drop $u$ and $\rho^{-1}$ for the following notations of transition probabilities. We have the following equality
where $\nabla\rho^{-1}$ is negligible 
\[
p_{\mu}(0,\xi,t,x)=\Sigma_{\mu\rho^{-1};u}^{*}(\eta,0;x,t)=\Sigma_{\mu\rho^{-1};-u,-\phi}(x,t;\xi,0).
\]

By applying {[}\citep{Friedman 1964}, Theorem 12 on page 25{]} to Proposition \ref{prop:theta mod} for $\omega^{\varepsilon}(x,t)$,
$\theta^{\varepsilon}(x,t)$ and $\phi^{\varepsilon}(x,t)$ for $d=3$ we have 
\begin{align*}
\omega^{\varepsilon}(x,t) & =\int_{D}p_{\mu}^{D}(0,\xi,t,x)C(\xi,t)\cdot\omega^{\varepsilon}(\xi,0)\mathrm{d}\eta+\int_{0}^{t}\int_{D}p_{\mu}^{D}(s,\xi,t,x)C(\xi,s)\cdot\nabla\wedge F(\xi,s)\mathrm{d}\xi\mathrm{d}s\\
 & +\int_{0}^{t}\int_{D}p_{\mu}^{D}(s,\xi,t,x)C(\xi,s)\cdot\chi_{\varepsilon}(\xi,s)\mathrm{d}\xi\mathrm{d}s,
\end{align*}
and 
\[
\Theta^{\varepsilon}(x,t)=\int_{D}p_{\kappa}^{D}(0,\xi,t,x)\check{C}(\xi,t)\cdot\Theta^{\varepsilon}(\xi,0)\mathrm{d}\xi+\int_{0}^{t}\int_{D}p_{\kappa}^{D}(s,\xi,t,x)\check{C}(\xi,s)\cdot\alpha_{\varepsilon}(\xi,s)\mathrm{d}\eta\mathrm{d}s;
\]
where $C(\eta,t)=C(\eta,t;0)$, $\check{C}(\eta,t)=\check{C}(\eta,t;0)$ and 
\begin{align*}
\frac{\mathrm{d}}{\mathrm{d}t}C(\xi,t;0)-C(\xi,t;0)\cdot S(\xi,t) & =0,\quad C_{j}^{i}(\xi,0;0)=\delta_{j}^{i},\\
\frac{\mathrm{d}}{\mathrm{d}t}\check{C}(\xi,t;0)-\check{C}(\xi,t;0)\cdot\bar{S}(\xi,t) & =0,\quad\check{C}_{j}^{i}(\xi,0;0)=\delta_{j}^{i}.
\end{align*}
Also we have 
\begin{align*}
\phi^{\varepsilon}(x,t)=\int_{D}p_{\mu+\lambda}^{D}(0,\xi,t,x)\phi^{\varepsilon}(\xi,0)\mathrm{d}\xi+\int_{0}^{t}\int_{D}p_{\mu+\lambda}^{D}(s,\xi,t,x)\vartheta_{\varepsilon}(\xi,s)\mathrm{d}\xi\mathrm{d}s.
\end{align*}

In particular, for $d=2$ we have the similar representation for $\Theta^{\varepsilon}(x,t)$ and $\phi^{\varepsilon}(x,t)$,
with $\omega^{\varepsilon}(x,t)$ is simplified to be 
\begin{align*}
\omega^{\varepsilon}(x,t) & =\int_{D}p_{\mu}^{D}(0,\xi,t,x)\omega^{\varepsilon}(\xi,0)\mathrm{d}\xi+\int_{0}^{t}\int_{D}p_{\mu}^{D}(s,\xi,t,x)\nabla\wedge F(\xi,s)\mathrm{d}\xi\mathrm{d}s\\
 & +\int_{0}^{t}\int_{D}p_{\mu}^{D}(s,\xi,t,x)\chi_{\varepsilon}(\xi,s)\mathrm{d}\xi\mathrm{d}s.
\end{align*}
Since the case where $d=2$ is just a special case of three dimensional case, we will only illustrate the proofs where $d=3$. 
\begin{defn}
If $\varphi$ is a Borel measurable function on $\mathbb{R}_{+}^{d}$, then for every $x\in\mathbb{R}^{d}$ we define 
\[
\hat{\varphi}(x)=1_{\{x_{d}>0\}}\varphi(x)-1_{\{x_{d}<0\}}\varphi(\bar{x}).
\]
\end{defn}

\begin{lem}
Let $D=\mathbb{R}_{+}^{d}$ and $\lambda>0$. Then for every $x\in D$ and $t>s\geq0$ we have 
\[
\int_{D}p_{\mu}^{D}(s,\xi,t,x)\varphi(\xi)\mathrm{d}\xi=\int_{\mathbb{R}^{d}}p_{\mu}(s,\xi,t,x)\hat{\varphi}(\xi)\mathrm{d}\xi.
\]
\end{lem}

Hence we are able to derive the following representations. 
\begin{lem}
For every $\varepsilon>0$ we have 
\begin{align*}
\omega^{\varepsilon}(x,t) & =\int_{D}p_{\mu}^{D}(0,\xi,t,x)C(\xi,t)\cdot\omega^{\varepsilon}(\xi,0)\mathrm{d}\xi+\int_{0}^{t}\int_{D}p_{\mu}^{D}(s,\xi,t,x)C(\xi,s)\cdot\nabla\wedge F(\xi,s)\mathrm{d}\xi\mathrm{d}s\\
 & +\int_{0}^{t}\int_{D}p_{\mu}^{D}(s,\xi,t,x)C(\xi,s)\cdot\chi_{\varepsilon}(\xi,s)\mathrm{d}\xi\mathrm{d}s,
\end{align*}
\[
\Theta^{\varepsilon}(x,t)=\int_{D}p_{\kappa}^{D}(0,\xi,t,x)\check{C}(\xi,t)\cdot\Theta^{\varepsilon}(\xi,0)\mathrm{d}\xi+\int_{0}^{t}\int_{D}p_{\kappa}^{D}(s,\xi,t,x)\check{C}(\xi,s)\cdot\alpha_{\varepsilon}(\xi,s)\mathrm{d}\xi\mathrm{d}s
\]
and
\[
\phi^{\varepsilon}(x,t)=\int_{D}p_{\mu+\lambda}^{D}(0,\xi,t,x)\phi^{\varepsilon}(\xi,0)\mathrm{d}\xi+\int_{0}^{t}\int_{D}p_{\mu+\lambda}^{D}(s,\xi,t,x)\vartheta_{\varepsilon}(\xi,s)\mathrm{d}\xi\mathrm{d}s
\]
where 
\begin{align*}
\frac{\mathrm{d}}{\mathrm{d}t}C(\xi,t;0)-C(\xi,t;0)\cdot S(\xi,t) & =0,\quad C_{j}^{i}(\xi,0;0)=\delta_{j}^{i},\\
\frac{\mathrm{d}}{\mathrm{d}t}\check{C}(\xi,t;0)+\check{C}(\xi,t;0)\cdot\overline{S}(\xi,t) & =0,\quad\check{C}_{j}^{i}(\xi,0;0)=\delta_{j}^{i}.
\end{align*}
for every $x\in\mathbb{R}_{+}^{d}$ and $t>0$. 
\end{lem}

By letting $\varepsilon\downarrow0$, we have the following theorem. 
\begin{thm}
\label{thmA.1}The following integral representations hold: 
\begin{align}
\omega(x,t) & =1_{\{x_{3}=0\}}\omega_{b}(x_{1},x_{2},t)+\int_{\mathbb{R}^{3}}p_{\mu}(0,\xi,t,x)C(\xi,t)\cdot\hat{\omega}(\xi,0)\mathrm{d}\xi\nonumber \\
 & \quad+\int_{0}^{t}\int_{\mathbb{R}^{3}}p_{\mu}(s,\xi,t,x)C(\xi,s)\cdot\hat{f}(\xi,s)\mathrm{d}\xi\mathrm{d}s\nonumber \\
 & \quad+2\mu\int_{0}^{t}\int_{\{\xi_{3}=0\}}\frac{\partial}{\partial\xi_{3}}\Bigg|_{\xi_{3}=0}\left[p_{\mu}(s,\xi,t,x)C(\xi,s)\right]\cdot\omega_{b}(\xi,s)\mathrm{d}\xi_{1}\mathrm{d}\xi_{2}\mathrm{d}s,\label{eq:omega_new}
\end{align}
\begin{align}
\Theta(x,t) & =1_{\{x_{3}=0\}}\Theta_{b}(x_{1},x_{2},t)+\int_{\mathbb{R}^{3}}p_{\kappa}(0,\xi,t,x)\check{C}(\xi,t)\cdot\hat{\Theta}(\xi,0)\mathrm{d}\xi\nonumber \\
 & \quad+2\kappa\int_{0}^{t}\int_{\{\xi_{3}=0\}}\frac{\partial}{\partial\xi_{3}}\Bigg|_{\xi_{3}=0}\left[p_{\kappa}(s,\xi,t,x)\check{C}(\xi,s)\right]\cdot\Theta_{b}(\xi,s)\mathrm{d}\xi_{1}\mathrm{d}\xi_{2}\mathrm{d}s\label{eq:vartheta_new}
\end{align}
and
\begin{align}
\phi(x,t) & =1_{\{x_{3}=0\}}\phi_{b}(x_{1},x_{2},t)+\int_{\mathbb{R}^{3}}p_{\mu+\lambda}(0,\xi,t,x)\hat{\phi}(\xi,0)\mathrm{d}\xi\nonumber \\
 & \quad+2(\mu+\lambda)\int_{0}^{t}\int_{\{\xi_{3}=0\}}\phi_{b}(\xi,s)\frac{\partial}{\partial\xi_{3}}\Bigg|_{\xi_{3}=0}p_{\mu+\lambda}(s,\xi,t,x)\mathrm{d}\xi_{1}\mathrm{d}\xi_{2}\mathrm{d}s,\label{eq:phi_new}
\end{align}
for $x\in\mathbb{R}_{+}^{3}$ and $f\coloneqq\nabla\wedge F$. 
\end{thm}

\begin{proof}
We will prove (\ref{eq:omega_new}) in detail, the other arguments are similar. We have 
\[
\lim_{\varepsilon\rightarrow0+}\int_{\mathbb{R}^{3}}p_{\mu}(0,\xi,t,x)C(\xi,t)\cdot\hat{\omega^{\epsilon}}(\xi,0)\mathrm{d}\xi=\int_{\mathbb{R}^{3}}p_{\mu}(0,\xi,t,x)C(\xi,t)\cdot\hat{\omega}(\xi,0)\mathrm{d}\xi.
\]
Hence we only need to deal with the limit of error term 
\[
J(x,\varepsilon)=\int_{0}^{t}\int_{\mathbb{R}^{3}}p_{\mu}(s,\xi,t,x)C(\xi,s)\cdot\hat{\chi}_{\varepsilon}(\xi,s)\mathrm{d}\xi\mathrm{d}s
\]
where $\chi_{\varepsilon}$ is given in (\ref{eq:chi}). Since when we let $\varepsilon\downarrow0$ the contribution of
the first term vanishes, therefore we only need to calculate 
\[
E_{1}^{\varepsilon}(x,t)=-\frac{1}{\varepsilon}\phi'\left(\frac{x_{3}}{\varepsilon}\right)\omega_{b}(x_{1},x_{2},t)\cdot C(x,t)u^{3}(x,t)
\]
and 
\[
E_{2}^{\varepsilon}(x,t)=\mu\frac{1}{\varepsilon^{2}}\phi''\left(\frac{x_{3}}{\varepsilon}\right)\omega_{b}(x_{1},x_{2},t)\cdot C(x,t).
\]
Let 
\[
J_{i}(x,\varepsilon)=\int_{0}^{t}\int_{\mathbb{R}^{3}}p_{\mu}(s,\xi,t,x)C(x,t)\cdot\hat{E_{i}^{\varepsilon}}(\xi,s)\mathrm{d}\xi\mathrm{d}s
\]
where $i=1,2$. By no-slip condition and integration by part we have 
\begin{align*}
I_{1}(\varepsilon) & \coloneqq\int_{0}^{\infty}p_{\mu}(s,\xi,t,x)E_{1}^{\varepsilon}(\xi,s)\mathrm{d}\xi_{3}\\
 & =-\omega_{b}(\xi_{1},\xi_{2},t)\int_{0}^{\varepsilon}u^{3}(\xi,s)p_{\mu}(s,\xi,t,x)C(\xi,s)\cdot\frac{1}{\varepsilon}\psi'\left(\frac{\xi_{2}}{\varepsilon}\right)\mathrm{d}\xi_{3}\\
 & =\omega_{b}(\xi_{1},\xi_{2},t)\int_{0}^{\varepsilon}\frac{\partial}{\partial\xi_{3}}\left[u^{3}(\xi,s)p_{\mu}(s,\xi,t,x)C(\xi,s)\right]\cdot\psi\left(\frac{\xi_{2}}{\varepsilon}\right)\mathrm{d}\xi_{3}\\
 & \rightarrow0\quad\textrm{as}\;\varepsilon\downarrow0.
\end{align*}
Hence we have 
\[
\lim_{\varepsilon\rightarrow0+}J_{1}(x,\varepsilon)=\lim_{\varepsilon\rightarrow0+}\int_{0}^{t}\int_{\mathbb{R}^{3}}p_{\mu}(s,\xi,t,x)\hat{E_{i}^{\varepsilon}}(\xi,s)\mathrm{d}\xi\mathrm{d}s=0.
\]

Now we consider $J_{2}(x,\varepsilon)$. We notice that 
\[
J_{2}(x,\varepsilon)=\int_{0}^{t}\int_{\mathbb{R}^{2}}\bigg[\int_{0}^{\infty}(p_{\mu}(s,\xi,t,x)-p_{\mu}(s,\bar{\xi},t,x))E_{i}^{\varepsilon}(\xi,s)\mathrm{d}\xi_{3}\bigg]\mathrm{d}\xi_{2}\mathrm{d}\xi_{1}\mathrm{d}s
\]
and again by no-slip condition and integration by part we have 
\begin{align*}
I_{2}(\varepsilon) & \coloneqq\int_{0}^{\infty}\left(p_{\mu}(s,\xi,t,x)-p_{\mu}(s,\bar{\xi},t,x)\right)E_{i}^{\varepsilon}(\xi,s)\mathrm{d}\xi_{3}\\
 & =\mu\omega_{b}(\xi_{1},\xi_{2},t)\int_{0}^{\varepsilon}\left(p_{\mu}(s,\xi,t,x)-p_{\mu}(s,\bar{\xi},t,x)\right)C(x,t)\cdot\frac{1}{\varepsilon^{2}}\psi''\left(\frac{x_{3}}{\varepsilon}\right)\mathrm{d}\xi_{3}\\
 & =-\mu\omega_{b}(\xi_{1},\xi_{2},t)\int_{0}^{\varepsilon}\frac{\partial}{\partial\eta_{3}}\left[(p_{\mu}(s,\xi,t,x)-p_{\mu}(s,\bar{\xi},t,x))C(x,t)\right]\cdot\frac{1}{\varepsilon}\psi'\left(\frac{x_{3}}{\varepsilon}\right)\mathrm{d}\xi_{3}\\
 & =2\mu\omega_{b}(\xi_{1},\xi_{2},t)\frac{\partial}{\partial\xi_{3}}\Bigg|_{\xi_{3}=0}\left[(p_{\mu}(s,\xi,t,x)C(x,t)\right]\\
 & \quad+\mu\omega_{b}(\xi_{1},\xi_{2},t)\int_{0}^{\varepsilon}\frac{\partial^{2}}{\partial\xi_{3}^{2}}\left[(p_{\mu}(s,\xi,t,x)-p_{\mu}(s,\bar{\xi},t,x))C(x,t)\right]\cdot\psi\left(\frac{x_{3}}{\varepsilon}\right)\mathrm{d}\xi_{3}.
\end{align*}
Then we have 
\[
I_{2}(\varepsilon)\rightarrow2\mu\sigma(\xi_{1},\xi_{2},t)\frac{\partial}{\partial\xi_{3}}\Bigg|_{\xi_{3}=0}\left[p_{\mu}(s,\xi,t,x)C(x,t)\right]\quad as\;\varepsilon\downarrow0.
\]
Hence we can conclude that 
\[
\lim_{\varepsilon\rightarrow0+}J_{2}(x,\varepsilon)=2\mu\int_{0}^{t}\int_{\{\xi_{3}=0\}}\frac{\partial}{\partial\xi_{3}}\Bigg|_{\xi_{3}=0}\left[p_{\mu}(s,\xi,t,x)C(\xi,s)\right]\cdot\omega_{b}(\eta,s)\mathrm{d}\xi_{1}\mathrm{d}\xi_{2}\mathrm{d}s.
\]

By the Biot-Savart law, we have the following theorem. 
\end{proof}
\begin{thm}
\label{thmA.2}Let $X$, $Y$ and $Z$ be defined as in (\ref{sde X-1-2}), (\ref{sde Y-1-2}) and (\ref{sde Z-1-2}). The
following functional integral representations hold: 
\begin{align*}
u(x,t) & =\int_{\mathbb{R}^{3}}\mathbb{E}\bigg[1_{\mathbb{R}_{+}^{3}}(X_{t}^{\eta})\Lambda_{2}(X_{t}^{\eta},x)\wedge(C(x,t)\cdot\hat{\omega}(\eta,0))\bigg]\mathrm{d}\eta\\
 & \quad+\int_{0}^{t}\int_{\mathbb{R}^{3}}\mathbb{E}\bigg[1_{\mathbb{R}_{+}^{3}}(X_{t}^{\eta,s})\Lambda_{3}(X_{t}^{\eta,s},x)\wedge(C(x,t)\cdot\hat{f}(\eta,s))\bigg]\mathrm{d}\eta\mathrm{d}s\\
 & \quad+2\mu\int_{0}^{t}\int_{\{\eta_{3}=0\}}\frac{\partial}{\partial\eta_{3}}\Bigg|_{\eta_{3}=0}\mathbb{E}\bigg[1_{\mathbb{R}_{+}^{3}}(X_{t}^{\eta,s})\Lambda_{3}(X_{t}^{\eta,s},x)\wedge(C(\eta,s)\cdot\omega_{b}(\eta_{1},\eta_{2},s))\bigg]\mathrm{d}\eta_{1}\mathrm{d}\eta_{2}\mathrm{d}s\\
 & \quad-\int_{\mathbb{R}^{3}}\mathbb{E}\bigg[\Lambda_{3}(Z_{t}^{\xi,\tau},x)\hat{\phi}(\xi,0)\big]\mathrm{d}\xi\\
 & \quad-2(\mu+\lambda)\int_{0}^{t}\int_{\{\xi_{3}=0\}}\frac{\partial}{\partial\xi_{3}}\Bigg|_{\xi_{3}=0}\mathbb{E}\bigg[\Lambda_{3}(Z_{t}^{\xi,\tau},x)\wedge\phi_{b}(\xi,s)\bigg]\mathrm{d}\xi_{1}\mathrm{d}\xi_{2}\mathrm{d}s.
\end{align*}
for every $x\in\mathbb{R}_{+}^{2}$ and $t>s\geq0$. 
\end{thm}

\begin{proof}
By Biot-Savart law and (\ref{eq:omega_new}), we have 
\begin{align*}
u(x,t) & =\int_{\mathbb{R}_{+}^{3}}\Lambda_{3}(y,x)\wedge\omega(y,t)\mathrm{d}y-\int_{\mathbb{R}_{+}^{3}}\phi(y)\Lambda_{3}(y,x)\mathrm{d}y\\
 & =\int_{\mathbb{R}_{+}^{3}}\Lambda_{3}(y,x)\wedge1_{\{y_{3}=0\}}\omega_{b}(y_{1},y_{2},t)\mathrm{d}y\\
 & \quad+\int_{\mathbb{R}^{3}}\int_{\mathbb{R}_{+}^{3}}\Lambda_{3}(y,x)\wedge\left[C(\eta,t)\cdot\hat{\omega}(\eta,0)p_{\mu}(0,\eta,t,y)\right]\mathrm{d}y\mathrm{d}\eta\\
 & \quad+\int_{0}^{t}\int_{\mathbb{R}^{3}}\int_{\mathbb{R}_{+}^{3}}\Lambda_{3}(y,x)\wedge\left[C(\eta,s)\cdot\hat{f}(\eta,s)p_{\mu}(s,\eta,t,y)\right]\mathrm{d}y\mathrm{d}\eta\mathrm{d}s\\
 & \quad+2\mu\int_{0}^{t}\int_{\{\eta_{3}=0\}}\frac{\partial}{\partial\eta_{3}}\Bigg|_{\eta_{3}=0}\int_{\mathbb{R}_{+}^{3}}\Lambda_{3}(y,x)\wedge\left[p_{\mu}(s,\eta,t,y)C(\eta,s)\cdot\omega_{b}(\eta_{1},\eta_{2},s)\right]\mathrm{d}y\mathrm{d}\eta_{1}\mathrm{d}\eta_{2}\mathrm{d}s\\
 & \quad-\int_{\mathbb{R}_{+}^{3}}\Lambda_{3}(y,x)1_{\{y_{3}=0\}}\phi_{b}(y_{1},y_{2},t)\mathrm{d}y\\
 & \quad-\int_{\mathbb{R}^{3}}\int_{\mathbb{R}_{+}^{3}}\Lambda_{3}(y,x)p_{\mu+\lambda}(0,\xi,t,y)\hat{\phi}(\xi,0)\mathrm{d}y\mathrm{d}\xi\\
 & \quad-2(\mu+\lambda)\int_{0}^{t}\int_{\{\xi_{3}=0\}}\frac{\partial}{\partial\xi_{3}}\Bigg|_{\xi_{3}=0}\int_{\mathbb{R}_{+}^{3}}\Lambda_{3}(y,x)\wedge\left[\phi_{b}(\xi,s)p_{\mu+\lambda}(s,\xi,t,x)\right]\mathrm{d}\xi_{1}\mathrm{d}\xi_{2}\mathrm{d}s,
\end{align*}
hence we get 
\begin{align*}
u(x,t) & =\int_{\mathbb{R}^{3}}\mathbb{E}\bigg[1_{\mathbb{R}_{+}^{3}}(X_{t}^{\eta})\Lambda_{2}(X_{t}^{\eta},x)\wedge(C(x,t)\cdot\hat{\omega}(\eta,0))\bigg]\mathrm{d}\eta\\
 & \quad+\int_{0}^{t}\int_{\mathbb{R}^{3}}\mathbb{E}\bigg[1_{\mathbb{R}_{+}^{3}}(X_{t}^{\eta,s})\Lambda_{3}(X_{t}^{\eta,s},x)\wedge(C(x,t)\cdot\hat{f}(\eta,s))\bigg]\mathrm{d}\eta\mathrm{d}s\\
 & \quad+2\mu\int_{0}^{t}\int_{\{\eta_{3}=0\}}\frac{\partial}{\partial\eta_{3}}\Bigg|_{\eta_{3}=0}\mathbb{E}\bigg[1_{\mathbb{R}_{+}^{3}}(X_{t}^{\eta,s})\Lambda_{3}(X_{t}^{\eta,s},x)\wedge(C(\eta,s)\cdot\omega_{b}(\eta_{1},\eta_{2},s))\bigg]\mathrm{d}\eta_{1}\mathrm{d}\eta_{2}\mathrm{d}s\\
 & \quad-\int_{\mathbb{R}^{3}}\mathbb{E}\bigg[\Lambda_{3}(Z_{t}^{\xi,\tau},x)\hat{\phi}(\xi,0)\bigg]\mathrm{d}\xi\\
 & \quad-2(\mu+\lambda)\int_{0}^{t}\int_{\{\xi_{3}=0\}}\frac{\partial}{\partial\xi_{3}}\Bigg|_{\xi_{3}=0}\mathbb{E}\bigg[\Lambda_{3}(Z_{t}^{\xi,\tau},x)\wedge\phi_{b}(\xi,s)\bigg]\mathrm{d}\xi_{1}\mathrm{d}\xi_{2}\mathrm{d}s.
\end{align*}
\end{proof}

\section*{Data Availability Statement}

No data are used in this article to support the findings of this study.

\section*{Declaration of Interests}

The authors report no conflict of interest.

\section*{Acknowledgement}

This work is supported partially by the EPSRC Centre for Doctoral Training in Mathematics of Random Systems: Analysis, Modelling
and Simulation (EP/S023925/1).

%\bibitem{ChenFeldman2010} Chen, G.-Q. G. and Feldman, M. 2010 Global
%solutions to shock reflection by large-angle wedges for potential
%flow. \emph{Ann. of Math}. $\mathbf{171}$: 1019-1134.
%\bibitem{ChenFeldman2018} Chen, G.-Q. G. and Feldman, M. 2018 \emph{The
%mathematics of shock reflection-diffraction and von Neumann's conjectures}.
%Ann. of Math. Studies. Princeton University Press.

\end{document}